\documentclass[reqno,twoside]{amsart}
\usepackage{amsmath}
\usepackage{amsfonts}
\usepackage{amssymb}
\usepackage{enumerate}

\usepackage{a4wide,amsmath,amssymb,latexsym,amsthm}

\numberwithin{equation}{section}
\usepackage{mathrsfs,mathtools,epic,bm}
\usepackage{hyperref}
\hypersetup{colorlinks=true,linkcolor=blue,filecolor=mangeta,urlcolor= cyan}

\newtheorem{theorem}{Theorem}[section]
\newtheorem{proposition}[theorem]{Proposition}
\newtheorem{lemma}[theorem]{Lemma}

\theoremstyle{definition}
\newtheorem{definition}[theorem]{Definition}

\newtheorem{remark}[theorem]{Remark}

\newtheorem{assumption}[theorem]{Assumption}

\numberwithin{equation}{section}

\def \dis {\displaystyle}

\def \R {\mathbb{R}}

\def \N {\mathbb{N}}

\def \V {\cal V}

\def \U {{\mathcal U}}

\def \hvarphi \hat{\varphi}
\def \dx{\mathrm{d}x}
\def \dy {\mathrm{d}y}
\def \dq {\mathrm{d}x \, \mathrm{d}t}

\def\Om{\Omega}

\def \R {\mathbb{R}}

\def \N {\mathbb{N}}

\def \V {\mathbb{ V}}

\def \U {\mathcal{U}}

\keywords{Nonlinear fractional equation,  optimal solutions,  Pontryagin principle,  First- and second-order optimality conditions.}

\begin{document}
	\title[Optimal control of a nonlinear fractional elliptic equation]{Optimal control of a class of semilinear fractional elliptic equations}

	\author{Cyrille Kenne}
\address{Cyrille Kenne, Laboratoire des Mat\'eriaux et Mol\'ecules en Milieu Agressif (L3MA UR$4\_1$), UFR STE et IUT, Universit\'e des Antilles, 97275 Schoelcher, Martinique.}
\email{kenne853@gmail.com}
\author{Gis\`{e}le Mophou}
\address{Gis\`ele Mophou, Laboratoire L.A.M.I.A., D\'epartement de Math\'ematiques et Informatique, Universit\'{e} des Antilles, Campus Fouillole, 97159 Pointe-\`a-Pitre,(FWI), Guadeloupe, Laboratoire  MAINEGE, Universit\'e Ouaga 3S, 06 BP 10347 Ouagadougou 06, Burkina Faso.}
\email{gisele.mophou@univ-antilles.fr}
\author{Mahamadi Warma}
\address{Mahamadi Warma, Department of Mathematical Sciences and the Center for Mathematics and Artificial Intelligence (CMAI), George Mason University,  Fairfax, VA 22030, USA.}
\email{mwarma@gmu.edu}
\thanks{The third  author is partially supported by  US Army Research Office (ARO) under Award NO: W911NF-20-1-0115}
	
	
	
	\begin{abstract}
	In this paper, a class of semilinear fractional elliptic equations associated to the spectral fractional Dirichlet Laplace operator is considered. We establish the existence of optimal solutions as well as a minimum principle of Pontryagin type and the first order necessary optimality conditions of associated optimal control problems.  Second order conditions for optimality are also obtained for $L^{\infty}$ and $L^2-$ local solutions under some structural assumptions.
\end{abstract}
	\maketitle	
	
	\section{Introduction}
Let $\Om\subset\R^N$ ($N\geq 2$) be an open and bounded domain with  boundary $\partial \Om$. In this paper, we are interested in the following control problem: Find
\begin{equation}\label{opt}
\inf_{u\in \U} J(u):=\int_{\Omega}L(x,y(x),u(x))\,\dx,
\end{equation}
subject to the constraints that $y$ solves the semilinear fractional elliptic diffusion equation
\begin{equation}\label{model}
\left\{
\begin{array}{rllll}
(-\Delta_D)^sy&=& F(x,y,u) \qquad &\hbox{in}& \Omega,\\
y&=&0  &\hbox{on}& \partial \Omega,
\end{array}
\right.
\end{equation}
and the set of admissible controls is given by	
\begin{equation}\label{defuad}
\U:=\big\{v\in L^\infty(\Omega): \alpha\leq v(x)\leq \beta \mbox{ for a.e.  }  x\in\Omega, \, \alpha,\beta\in \R,\, \beta>\alpha\big\}.
\end{equation}
In \eqref{opt}, $L:\Omega\times\R\times\R\to \R$ is a given function satisfying suitable conditions.
In \eqref{model}, $y$ denotes the state, $u$ is the control function, $(-\Delta_D)^s$ denotes the spectral fractional Dirichlet Laplace operator of order $0<s<1$,  that is, the fractional $s$ powers of the realization in $L^2(\Omega)$ of the Laplace operator $-\Delta$ with zero Dirichlet boundary condition on $\partial \Omega$ and  $F:\Omega\times\R\times\R\to \R$ is a given function satisfying suitable conditions.  We notice that the boundary condition in \eqref{model} can be dropped since it is already included in the definition
of the operator $(-\Delta_D)^s$.  The precise assumptions on the functions $L$ and $F$ will be given in Section \ref{sec-as}.

The main purpose of this paper is to solve the optimal control problem \eqref{opt}-\eqref{model}. The control appears nonlinearly in the state equation and the well known classical methods do not apply.  
 The optimal control of nonlinear systems with control appearing nonlinearly in the state equation has been considered before by many researchers (see e.g. \cite{casas2000,casas1995,casas1997,eppler1997, raymond1998} and their references).
In \cite{raymond1999}, J.P. Raymond and H. Zidani addressed the optimal control of problems governed by semilinear parabolic equations. The controls appeared in a nonlinear form in the state equation. After some regularity results, they obtained using the Hamiltonian Pontryagin's principle the optimality conditions. In \cite{casas2020o}, E. Casas and F. Tr\"{o}ltzsch considered the optimal control of a semilinear elliptic equation with Neumann boundary condition in space dimension $N=2$ or $N=3$ where the control appears nonlinearly in the sate equation. Let us also mention that their control was not explicitly included in the objective functional. Using a measurable selection technique, they proved the existence of optimal controls. They also derived the first and second orders conditions for optimality.  The case of second order elliptic operators with a Dirichlet boundary condition was previously studied by Casas and Yong \cite{casas1995}.
All the above mentioned papers were dealing with integer order operators, that is, classical second order elliptic operators with standard boundary conditions (Dirichlet and Neumann).

In the case of fractional order operators, only few papers considered it to the best of our knowledge. In \cite{kien2022},  B. Kien et \textit{al.} considered a fractional differential equation of order $0<\alpha<1$ with control constraints. The control appeared in a nonlinear form in the state equation and they established the first and second order optimality conditions for local optimal solutions. A theory of no-gap second-order conditions were also obtained in the case $1/2<\alpha<1$.
R. Kamocki \cite{kamocki2021} considered an optimal control problem containing a control system described by a nonlinear partial differential equation with the fractional Dirichlet-Laplacian associated to an integral cost. He proved by means of the Filippov's theorem the existence of optimal solutions.  The papers \cite{AntilWarma,otarola2022} considered similar problems where the control appears linearly in the state equation and the underlying operator can be the fractional power of second order uniformly elliptic operators with Dirichlet boundary condition or the integral fractional Laplace operator with zero Dirichlet exterior condition.  The linear case with non zero Dirichlet exterior condition has been first investigated in \cite{AKW}.

In this paper, we consider the optimal control problem \eqref{opt}-\eqref{model}. The system is nonlinear with respect to both the state $y$ and the control $u$. We prove some existence and regularity results for the state equation \eqref{model} and obtain the existence of optimal solutions to the control problem \eqref{opt}-\eqref{model}. The latter result is obtained by using a measurable selection theorem. Next, we derive the first order optimality conditions and obtain a pointwise Pontryagin's principle. Finally, we derive the second order conditions for optimality. 

The paper is organized as follows. In Section \ref{prelim}, we introduce the function spaces and state the assumptions needed in the sequel. Section \ref{existe} deals with the existence and regularity results for the state equation \eqref{model} where the main results are obtained in Theorems \ref{theoexistence} and \ref{theoregularity}. In Section \ref{control}, we prove the existence of optimal solutions (Theorem \ref{existcontrol1}). Section \ref{sec-5} is devoted to the first order necessary optimality conditions that are obtained in Theorem \ref{theoSO} and the minimum principle of Pontryagin type proved in Proposition \ref{pr1}.
In Section \ref{second} we derive the second order necessary and sufficient conditions for local optimality.

	\section{Preliminary results and assumptions} \label{prelim}	
	For the sake of completeness,  we give some well-known results that are used throughout the paper and we state the assumptions on the nonlinearities $L$ and $F$.  Throughout the paper without any mention $\Omega\subset\mathbb R^N$ is an arbitrary bounded open set.  We shall specify if a regularity on $\Omega$ is needed.

\subsection{Preliminaries}
	Let
	$\mathcal{D}(\Omega)$ be the space of test functions on $\Omega$, that is, the space of all  infinitely continuously differentiable functions with compact support in $\Omega$. We let 
	$$H^1_0(\Omega)=\overline{\mathcal{D}(\Omega)}^{H^1(\Omega)}$$ 
	where
		$$H^1(\Omega)=\Big\{u\in L^2(\Omega):\int_{\Omega}|\nabla u|^2\,\dx<\infty\Big\}$$
	is the first order Sobolev space endowed with the norm
	$$\|u\|_{H^1(\Omega)}=\Big(\int_{\Omega}|u|^2\,\dx+\int_{\Omega}|\nabla u|^2\,\dx\Big)^{\frac{1}{2}}.$$	
Let $-\Delta_D$ be the realization on $L^2(\Omega)$ of the Laplace operator $-\Delta$ with the zero Dirichlet boundary condition. That is, $-\Delta_D$ is the positive and self-adjoint operator on $L^2(\Omega)$ associated with the closed, bilinear, coercive  and symmetric form
$$\mathcal{A}_D(u,v)=\int_{\Omega}\nabla u\cdot \nabla v\,\dx,\;\;\;u,v\in H^1_0(\Omega),$$
in the sense that
\begin{equation}\label{space}
\left\{
\begin{array}{llll}
D(\Delta_D)=\Big\{u\in H^{1}_0(\Omega):\,\exists w\in L^2(\Omega),\; \mathcal{A}_D(u,v)=(w,v)_{L^2(\Omega)},\;\;\forall v\in H^1_0(\Omega) \Big\},\\
-\Delta_D u=w.
\end{array}
\right.
\end{equation}
For instance if $\Omega$ has a smooth boundary,  say of class $C^2$,  then $D(\Delta_D)=H^2(\Omega)\cap H^1_0(\Omega)$, where
$$H^2(\Omega)=\Big\{u\in H^1(\Omega):\partial_{x_j}u\in H^1(\Omega),\; j=1,2,...,N\Big\}.$$	
It is well-known that $-\Delta_D$ has a compact resolvent and its eigenvalues form a non-decreasing sequence	$0<\lambda_1\leq\lambda_2\leq\cdots\leq \lambda_n\leq \cdots$ of real numbers satisfying $\dis \lim_{n\to \infty} \lambda_n=\infty$. We denote by $(\varphi_n)$ the orthonormal basis of eigenfunctions associated with the eigenvalues $(\lambda_n)$.
	
For $0<s<1$, we define the fractional order Sobolev space
	\begin{align*}
		H^s(\Omega):=\left\{u\in L^2(\Omega):\; \int_{\Omega}\int_{\Omega}\frac{|u(x)-u(y)|^2}{|x-y|^{N+2s}}\;dxdy<\infty\right\}
	\end{align*}
	and we endow it with the norm given by
	\begin{align*}
		\|u\|_{H^s(\Omega)}=\left(\int_{\Omega}|u|^2\;dx+\int_{\Omega}\int_{\Omega}\frac{|u(x)-u(y)|^2}{|x-y|^{N+2s}}\;dxdy\right)^{1/2}.
	\end{align*}
	We set
	$$
H^s_0(\Omega)=\overline{\mathcal{D}(\Omega)}^{H^s(\Omega)},
	$$
	and 
	$$H^\frac{1}{2}_{00}(\Omega):=\left\{ u\in H^{\frac{1}{2}}(\Omega): \dis \int_{\Omega}\frac{u^2(x)}{\mbox{dist}(x,\partial \Omega)}\,\dx<\infty\right\}$$
	where $\mbox{dist}(x,\partial \Omega)$, $x\in\Omega$, denotes the distance from $x$ to $\partial\Omega$.
	
We have that $H_0^s(\Omega)$ is a Hilbert space and if $1/2<s<1$, then the norm on $H_0^{s}(\Omega)$ is equivalent to the norm
	\begin{equation}\label{norm0bisinter}
		\|w\|_{H_0^{s}(\Omega)}:=\left(\int_{\Omega}\int_{\Omega}
		\frac{(w(x)-w(y))^2}{|x-y|^{N+2s}}\;\dx\, \dy\right)^{1/2}.
	\end{equation}
Since $\Omega$ is assumed to be bounded we have the following continuous embedding:
	\begin{equation}\label{equality0}
	H^s_0(\Omega)\hookrightarrow\left\{
	\begin{array}{lllll}
	L^{\frac{2N}{N-2s}}(\Omega)&\text{if}& N>2s,\\
	L^p(\Omega),\;\;p\in [1,\infty) &\text{if}& N=2s,\\
	\mathcal{C}^{0,s-\frac{N}{2}}(\bar{\Omega}) &\text{if}& N<2s.
	\end{array}
	\right.
	\end{equation}
For any $s \geq 0$, we also introduce the following fractional order Sobolev space:
$$\mathbb{ H}^s(\Omega):=\left\{ u=\dis \sum_{i=1}^{\infty}u_n\varphi_n\in L^2(\Omega): \dis \|u\|^2_{\mathbb{H}^s(\Omega)}:=\sum_{i=1}^{\infty}\lambda^s_nu^2_n<\infty \right\},$$	
where we recall that $(\lambda_n)$ are the eigenvalues of $-\Delta_D$ with associated normalized eigenfunctions $(\varphi_n)$ and 	
$$u_n:=(u,\varphi_n)_{L^2(\Omega)}=\int_{\Omega}u\varphi_n\, \dx.$$
It is nowadays well-known that
\begin{equation}\label{equality1}
\mathbb{H}^s(\Omega)=\left\{
\begin{array}{lllll}
H^s_0(\Omega)&\text{if}& \dis s\ne \frac{1}{2},\\
H^\frac{1}{2}_{00}(\Omega) &\text{if}& s=\dis \frac{1}{2},
\end{array}
\right.
\end{equation}
It follows from \eqref{equality1} that the embedding \eqref{equality0} holds with $H^s_0(\Omega)$ replaced by $\mathbb{H}^s(\Omega)$. For more details on fractional order Sobolev spaces we refer to \cite{nezza2012,grisvard2011,War} and their references.

\begin{definition}
The spectral fractional Dirichlet Laplacian of order $s\ge 0$ is defined  by 
$$ D((-\Delta_D)^s)=\mathbb{H}^s(\Omega),\quad (-\Delta_D)^su=\sum_{n=1}^{\infty}\lambda^s_nu_n\varphi_n\;\;\;\;\text{with}\;\; u_n=\int_{\Omega}u\varphi_n\,\dx.$$
\end{definition}	

We notice that in this case we have that
\begin{equation}\label{norm}
\|u\|_{\mathbb{H}^s(\Omega)}=\|(-\Delta_D)^{\frac{s}{2}}u\|_{L^2(\Omega)}.
\end{equation}	
Note that $\mathcal{D}(\Omega)\hookrightarrow\mathbb{H}^s(\Omega)\hookrightarrow L^2(\Omega)\hookrightarrow  \mathbb{ H}^{-s}(\Omega)$, so, the operator $(-\Delta_D)^s$ is unbounded, densely defined and with bounded
inverse $(-\Delta_D)^{-s}$ in $L^2(\Omega)$.	But it can also be viewed as a bounded operator from $\mathbb{ H}^s(\Omega)$ into its dual $\mathbb{ H}^{-s}(\Omega):=(\mathbb{ H}^s(\Omega))^\star$.

Next, we consider the following linear elliptic problem:
\begin{equation}\label{LE}
(-\Delta_D)^s\varphi+b(x)\varphi=f
\end{equation}
where we assume that $b\in L^\infty(\Omega)$ and is nonnegative.

\begin{definition}
A function $\varphi\in \mathbb H^s(\Omega)$ is called a weak solution of \eqref{LE},  if the equality
\begin{align}\label{A0}
\int_{\Omega}(-\Delta)^{s/2}\varphi (-\Delta)^{s/2}\phi\;dx +\int_{\Omega}b(x)\varphi\phi\;dx=\langle f,\phi\rangle_{\mathbb{ H}^{-s}(\Omega),\mathbb{ H}^{s}(\Omega)},
\end{align}
holds for every $\phi\in \mathbb H^s(\Omega)$.
\end{definition}

We have the following result of existence and regularity of weak solutions. We refer to \cite{antil2017,AntilWarma,RO-SJ-EX} and their references for the complete proof.

\begin{theorem}\label{theo-23}
For every $f\in \mathbb{ H}^{-s}(\Omega)$, Equation \eqref{LE} has a unique weak solution $\varphi\in  \mathbb H^s(\Omega)$ and there is a  constant   $C=C(N,s,\Omega)>0$ such that
\begin{align}\label{A1}
\|\varphi\|_{L^2(\Omega)}\le \|\varphi\|_{ \mathbb H^s(\Omega)}\le C\|f\|_{ \mathbb{ H}^{-s}(\Omega)}.
\end{align}
In addition,  the following assertions hold:
\begin{enumerate}

\item If $N<2s$, then $\varphi\in \mathcal{C}^{0,s-\frac{N}{2}}(\bar{\Omega})$ and there is a constant $C>0$ such that
\begin{align}\label{A1-1}
\|\varphi\|_{\mathcal{C}^{0,s-\frac{N}{2}}(\bar{\Omega})}\le C\|f\|_{ \mathbb{ H}^{-s}(\Omega)}.
\end{align}

\item If $N= 2s$, then $\varphi\in L^p(\Omega)$  for every $p\in [1,\infty)$ and there is a constant $C>0$ such that 
\begin{align}\label{A1-0}
\|\varphi\|_{L^p(\Omega)}\le C\|f\|_{ \mathbb{ H}^{-s}(\Omega)}.
\end{align}
\item If $f\in L^p(\Omega)$ for some $p>N/2s$, then $\varphi\in L^\infty(\Omega)$ and there is a constant $C>0$ such that
\begin{align}\label{A2}
\|\varphi\|_{L^\infty(\Omega)}\le C\|f\|_{L^p(\Omega)}.
\end{align}
\item If $f\in L^2(\Omega)$, $N>2s$  and $1\le r<\frac{N}{N-2s}$, then  $\varphi\in L^r(\Omega)$ and there is a constant $C>0$ such that
\begin{align}\label{A3}
\|\varphi\|_{L^r(\Omega)}\le C\|f\|_{L^1(\Omega)}.
\end{align}
\end{enumerate}
\end{theorem}

\begin{proof}
The existence and uniqueness of solutions together with the estimate \eqref{A1}  is a simple application of the classical Lax-Miligram Lemma.  Parts (a) and (b) follow from \eqref{A1} and  the continuous embedding \eqref{equality0}.
The proof of Part (c) can be found in  \cite{antil2017,AntilWarma} and the references therein. For Part (d) we refer to \cite[Proposition 1.4]{RO-SJ-EX}.
\end{proof}

\subsection{General assumptions on the nonlinearities}\label{sec-as}

	We make the following assumptions on the function $F$ involved in the state equation \eqref{model}.

\begin{assumption}\label{ass1}
	The measurable function $F:\Omega\times\mathbb{R}\times\mathbb R\rightarrow \mathbb{R}$ is a continuous function of class $\mathcal{C}^2$ with respect to the last two components, monotone non-increasing with respect to the second component and satisfies: 
	\begin{itemize}
\item $\dis F(\cdot,0,0)\in L^{\tilde{p}}(\Omega)$ for some $\tilde{p}>N/2s$ and 
\begin{equation}\label{as1}
\begin{array}{llll}
	\dis \frac{\partial F}{\partial t}(x,t,\xi)\leq 0\;\; \text{for a.e.}\; x\in  \Omega \;\;\text{and for all}\;\; (t,\xi)\in \R^2.
	\end{array}
\end{equation}
\item $\forall M>0$, $\exists C_{F,M}>0$ such that	
	\begin{equation}\label{as2}
	\begin{array}{llll}
	 \dis \sum_{1\leq i+j\leq 2}\left|\frac{\partial^{i+j} F}{\partial t^i\partial \xi^j}(x,t,u)\right|\leq C_{F,M}
	\end{array}
	\end{equation}
for a.e. $x\in \Omega$ and for all $(t,\xi)\in \R^2$, with $|t|\leq M$ and $ \xi\in [\alpha,\beta]$.
\item 	$\forall \varepsilon>0$ and  $M>0$, $\exists \rho>0$ such that, if $|t_1|, |t_2|\leq M$, $|t_1-t_2|<\rho$, $\xi_1,\xi_2\in [\alpha,\beta]$ with $|\xi_1-\xi_2|<\rho$, then
	\begin{equation}\label{as3}
	\begin{array}{llll}
	\dis \sum_{i+j= 2}\left|\frac{\partial^{i+j} F}{\partial t^i\partial \xi^j}(x,t_2,\xi_2)-\frac{\partial^{i+j} F}{\partial t^i\partial \xi^j}(x,t_1,\xi_1)\right|\leq \varepsilon\;\; \text{for a.a.}\; x\in  \Omega.
	\end{array}
	\end{equation}
	\item 	$\forall (x,t)\in \Omega\times \R$, the mapping $F(x,t,\cdot)$ satisfies the following convexity condition: For any $\xi,v\in [\alpha,\beta]$ and $\lambda\in [0,1]$, there exists $w\in [\alpha,\beta]$ such that 
	\begin{equation}\label{as4}
	\begin{array}{llll}
	\dis F(x,t,w)=\lambda F(x,t,\xi)+(1-\lambda)F(x,t,v).
	\end{array}
	\end{equation}
\end{itemize}
\end{assumption}

We conclude this section by giving the assumptions on the nonlinearity $L$ involved in the functional $J$ given in \eqref{opt}.
\begin{assumption}\label{ass2}
	The measurable function $L:\Omega\times\mathbb{R}\times \mathbb R\rightarrow \mathbb{R}$ is assumed to be a Carath\'eodory function of class $\mathcal{C}^2$ with respect to the last two components and satisfies the following:
	\begin{itemize}
		\item $\dis \frac{\partial L}{\partial t}(\cdot,0,0),\frac{\partial L}{\partial \xi}(\cdot,0,0)\in L^\infty(\Omega)$.
		\item There exist a function $\phi_0\in L^1(\Omega)$ and nonnegative constants $c_0$, $c_1$, $\beta_0$, $\beta_1$ such that
		\begin{equation}\label{as5}
		\begin{array}{llll}
		\dis\phi_0(x)-c_0|x|^{\beta_0}-c_1|x|^{\beta_1}\leq L(x,t,\xi)\;\; \text{for a.e.}\; x\in  \Omega \;\;\text{and for all}\;\; (t,\xi)\in \R^2.
		\end{array}
		\end{equation}
		\item $\forall M>0$, $\exists C_{L,M}>0$ such that	
		\begin{equation}\label{as6}
		\begin{array}{llll}
		\dis |L(x,t,\xi)|\leq C_{L,M}
		\end{array}
		\end{equation}
		and 	
		\begin{equation}\label{as7}
		\begin{array}{llll}
		\dis \sum_{1\leq i+j\leq 2}\left|\frac{\partial^{i+j} L}{\partial t^i\partial \xi^j}(x,t,\xi)\right|\leq C_{L,M}
		\end{array}
		\end{equation}
		for a.e. $x\in \Omega$ and for all $(t,\xi)\in \R^2$, with $|t|\leq M$  and $ \xi\in [\alpha,\beta]$.
		\item 	$\forall \varepsilon>0$ and  $M>0$, $\exists \rho>0$ such that, if $|t_1|, |t_2|\leq M$, $|t_1-t_2|<\rho$, $\xi_1,\xi_2\in [\alpha,\beta]$ with $|\xi_1-\xi_2|<\rho$, then
		\begin{equation}\label{as8}
		\begin{array}{llll}
		\dis \sum_{i+j= 2}\left|\frac{\partial^{i+j} L}{\partial t^i\partial \xi^j}(x,t_2,\xi_2)-\frac{\partial^{i+j} L}{\partial t^i\partial \xi^j}(x,t_1,\xi_1)\right|\leq \varepsilon\;\; \text{for a.a.}\; x\in  \Omega.
		\end{array}
		\end{equation}
		\item 	$\forall (x,t)\in \Omega\times \R$, the mapping $F(x,t,\cdot)$ satisfies the following convexity condition: For any $\xi,v\in [\alpha,\beta]$ and $\lambda\in [0,1]$, there exists $w\in [\alpha,\beta]$ such that 
		\begin{equation}\label{as9}
		\begin{array}{llll}
		\dis L(x,t,w)\leq\lambda L(x,t,\xi)+(1-\lambda)L(x,t,v).
		\end{array}
		\end{equation}
	\end{itemize}
\end{assumption}

	
	\section{Existence and regularity of solutions to the state equation}\label{existe}

	From now on, we simplify the notations by setting
	 \begin{equation}
		\mathbb{V}:=  \mathbb{H}^s(\Omega) \hbox{ and }\mathbb V^\star:= \mathbb{H}^{-s}(\Omega),
	\end{equation}
	and we let $ \left\langle\cdot,\cdot\right\rangle_{\V^\star,\V}$ denote the duality mapping between $\V^\star$ and $\V$.
	
	\begin{definition}\label{weaksolutionint}
		Let $u\in L^\infty(\Omega)$. We  say that
		$y\in \V$ is a weak solution of \eqref{model}, if the  equality
		\begin{equation}\label{Eq-Def31int}
			\begin{array}{ccccc}
				\dis  \mathcal{F}(y,\phi):=\int_{\Omega}(-\Delta_D)^{\frac{s}{2}}y (-\Delta_D)^{\frac{s}{2}}\phi\,\dx&=\dis\int_{\Omega} F(x,y,u) \phi \,\dx
			\end{array}
		\end{equation}
		holds, for every $\phi \in \V$.
	\end{definition}

We have the following existence and regularity result.

\begin{theorem}\label{theoexistence}
Let $s\in (0,1)$ and $\tilde{p}>N/2s$.  Suppose that  \eqref{as1} and \eqref{as2} hold. Then,  for every $u\in L^\infty(\Omega)$,  Equation \eqref{model} has a unique weak solution $y\in \V\cap L^\infty(\Omega)$ and there is a constant $C:=C(N,s,\Omega, \tilde{p})>0$ such that 
	\begin{equation}\label{est0}
\|y\|_{\V}+\|y\|_{L^\infty(\Omega)}\leq C\left( \|F(\cdot,0,0)\|_{L^{\tilde{p}}(\Omega)}+\|u\|_{L^\infty(\Omega)}\right).
	\end{equation}
Moreover, there is a  constant $C:=C(N,s,\Omega, {\tilde{p}},\alpha,\beta)>0$  such that 
\begin{equation}\label{est1}
\|y\|_{\V}+\|y\|_{L^\infty(\Omega)}\leq C\left(\|F(\cdot,0,0)\|_{L^{\tilde{p}}(\Omega)}+1\right), \;\;\forall u\in \U.
\end{equation}
\end{theorem}

\begin{proof}
	Let $u\in L^\infty(\Omega)$. We claim that $F(\cdot,0,u)\in L^{\tilde{p}}(\Omega)$. Using the Mean Value Theorem, we write
	\begin{eqnarray*}
		F(x,0,u(x))=F(x,0,0)+\frac{\partial F}{\partial u}(x,0,\epsilon(x)u(x))u(x),
	\end{eqnarray*} 
	where $\epsilon:\Omega\to [0,1]$ is a measurable function. Therefore, using \eqref{as2} with $M=1$, we deduce that 
	\begin{eqnarray*}
		|F(x,0,u(x))|&\leq& |F(x,0,0)|+\left|\frac{\partial F}{\partial u}(x,0,\epsilon(x)u(x))\right||u(x)|\\
		&\leq& |F(x,0,0)|+C_{F,1}\|u\|_{L^\infty(\Omega)}.
	\end{eqnarray*}
	Hence,
	\begin{equation}\label{est}
	\|F(\cdot,0,u)\|_{L^{\tilde{p}}(\Omega)} \leq C(N,s,\Omega, {\tilde{p}})\left( \|F(\cdot,0,0)\|_{L^{\tilde{p}}(\Omega)}+\|u\|_{L^\infty(\Omega)}\right),
	\end{equation}
and so the claim is proved.  Using similar arguments as in \cite[Section 3]{AntilWarma} or \cite[Theorem 3.1]{otarola2022},  we can deduce that  \eqref{model} has a unique weak solution $y\in \V\cap L^\infty(\Omega)$.  In addition, the following estimate holds:
\begin{equation}\label{est3}
\|y\|_{\V}+\|y\|_{L^\infty(\Omega)}\leq C\|F(\cdot,0,u)\|_{L^{\tilde{p}}(\Omega)}.
\end{equation}
Combining \eqref{est}-\eqref{est3} we get \eqref{est0}.	If $u\in \mathcal{U}$, we can also deduce \eqref{est1}. The proof is finished.
\end{proof}
Next, we have the following regularity result of weak solutions to \eqref{model}.

\begin{theorem}\label{theoregularity}
	Let $s\in (0,1)$ and $u\in L^\infty(\Omega)$. Assume that Assumption \ref{ass1} is satisfied. Then,  the following assertions hold:
	\begin{enumerate}
		\item Let $\Omega$  be of class $\mathcal{C}^{1}$. If $\frac{1}{2}<s<1$ and $\frac{N}{2s}<{\tilde{p}}<\frac{N}{2s-1}$, or if $0<s\leq \frac{1}{2}$ and ${\tilde{p}}<\infty$, then the weak solution $y$ of \eqref{model} belongs to $\mathcal{C}^{0,\sigma}(\bar{\Omega})$ with $\sigma=2s-\frac{N}{{\tilde{p}}}$.
		\item Let $s>\frac{1}{2}$ and ${\tilde{p}}>\frac{N}{2s-1}$. If $\Omega$ is of class $\mathcal{C}^{1,\sigma}$ for $0<\sigma:=2s-\frac{N}{{\tilde{p}}}-1<1$, then the weak solution $y$ of \eqref{model} belongs to $\mathcal{C}^{0,\sigma}(\bar{\Omega})$.
		\item  In both cases $(a)$ and $(b)$, there is a constant $C:=C(N,s,\Omega)>0$ such that
		\begin{equation}\label{est2}
		\|y\|_{\mathcal{C}^{0,\sigma}(\bar\Omega)}\leq C\left(\|y\|_{\V}+\|F(\cdot,y(\cdot),u(\cdot))\|_{L^{\tilde{p}}(\Omega)}\right).
		\end{equation} 
	\end{enumerate}
\end{theorem}
%

\begin{proof}
We can write the system \eqref{model} as 
\begin{equation}\label{eq}
(-\Delta_D)^sy=F(x,y,u)\;\;\;\text{in}\;\; \Omega.
\end{equation}
Let $y\in \V\cap L^\infty(\Omega)$ be the weak solution of \eqref{model} equivalently of \eqref{eq}. According to \cite{antil2017,grubb2016}, it suffixes to show that $F(\cdot,y(\cdot),u(\cdot))\in L^{\tilde{p}}(\Omega)$. Indeed, using \eqref{as2} and \eqref{est0}, we have from the Mean Value Theorem that there is a measurable function $\epsilon:\Omega\to [0,1]$ such that
\begin{align}\label{ss}
\dis |F(x,y(x),u(x))|\leq& \dis|F(x,0,u(x))|+\left|\frac{\partial F}{\partial y}(x,\epsilon(x)y(x),u(x))\right||y(x)|\notag\\
\leq& \dis |F(x,0,u(x))|+\left|\frac{\partial F}{\partial y}(x,\epsilon(x)y(x),u(x))\right|\|y\|_{L^\infty(\Omega)}\notag\\
\leq& \dis |F(x,0,u(x))|+ C_{F,M}M,
\end{align}
where $M=C(N,s,\Omega, p)\left( \|F(\cdot,0,0)\|_{L^{\tilde{p}}(\Omega)}+\|u\|_{L^\infty(\Omega)}\right)>0$.
Hence, using \eqref{est} and \eqref{ss}, we deduce that
\begin{equation}\label{imp}
\begin{array}{llll}
\dis \|F(\cdot,y(\cdot),u(\cdot))\|_{L^{\tilde{p}}(\Omega)}\leq 2 C_{F,M}M.
\end{array}
\end{equation}
Therefore, $F(\cdot,y(\cdot),u(\cdot))\in L^{\tilde{p}}(\Omega)$. Applying the results of \cite{antil2017,caffarelli2016,grubb2016}, we can deduce that Parts (a) and (b) hold.  It follows from \cite{caffarelli2016}  that
 \begin{equation}\label{est4}
\|y\|_{\mathcal{C}^{0,\sigma}(\bar\Omega)}\leq C\left(\|y\|_{L^2(\Omega)}+\|y\|_{\V}+\|F(\cdot,y(\cdot),u(\cdot))\|_{L^{\tilde{p}}(\Omega)}\right).
\end{equation} 
Thanks to \eqref{equality0}, we have the continuous embedding $\V\hookrightarrow L^2(\Omega)$. Hence,
\begin{equation}\label{est5}
\|y\|_{L^2(\Omega)}\leq C(N,s,\Omega)\|y\|_{\V}.
\end{equation}
Combining \eqref{est4}-\eqref{est5} leads us to \eqref{est2}. This completes the proof.
\end{proof}

\begin{remark}\label{estc0}
Let $s>\frac{1}{2}$ and ${\tilde{p}}>\frac{N}{2s-1}$. If $\Omega$ is of class $\mathcal{C}^{1,\sigma}$ for $0<\sigma:=2s-\frac{N}{{\tilde{p}}}-1<1$, then from \eqref{est0}, \eqref{est1} and \eqref{est2}, we can deduce the following estimates:
\begin{equation}\label{estc00}
\|y\|_{\V}+\|y\|_{\mathcal{C}^{0,\sigma}(\bar{\Omega})}\leq C\left(\|F(\cdot,0,0)\|_{L^{\tilde{p}}(\Omega)}+\|u\|_{L^\infty(\Omega)}\right) \;\;\forall u\in L^\infty(\Omega)
\end{equation}
with $C:=C(N,s,\Omega,\tilde{p})>0$ and
\begin{equation}\label{estc01}
\|y\|_{\V}+\|y\|_{\mathcal{C}^{0,\sigma}(\bar{\Omega})}\leq C\left(\|F(\cdot,0,0)\|_{L^{\tilde{p}}(\Omega)}+1\right) \;\;\forall u\in \U
\end{equation}
where $C:=C(N,s,\Omega,\tilde{p},\alpha,\beta)>0.$
\end{remark}

	\section{Existence of optimal solutions} \label{control}
	In this section, we are concerned with the existence of optimal solutions to the control problem \eqref{opt}-\eqref{model}.  The theorem on measurable selections of a measurable set-valued (see e. g. \cite{aubin2009}) will play a crucial role. 
	\begin{theorem}\label{existcontrol1}
Assume that $s>\frac{1}{2}$, ${\tilde{p}}>\frac{N}{2s-1}$, $\Omega$ is of class $\mathcal{C}^{1,\sigma}$ for $0<\sigma:=2s-\frac{N}{{\tilde{p}}}-1<1$ and that Assumptions \ref{ass1} and \ref{ass2} are satisfied. Then, there exists at least one solution $(\bar u,\bar y)\in \mathcal{U}\times(\V\cap
 \mathcal{C}(\bar{\Omega}))$ of the minimization problem \eqref{opt}-\eqref{model}.
	\end{theorem}
	
	\begin{proof} We first note that, from \eqref{as5} and \eqref{as6}, $\dis j:=\inf_{v\in\U}J(v)$ satisfies $\dis-\infty< j<+\infty$.
		Let $(v_k)_{k\geq1}\subset  \mathcal{U}$ be a minimizing sequence  such that
		$$\lim_{k\to \infty}J(v_k)= j. $$
	Since $\mathcal{U}$ is bounded in $L^\infty(\Omega)$, we have that the sequence $(v_k)_{k\geq1}$ is bounded.
		Since $y_{k}:=y(v_k)$ is the state associated to the control $v_k$,  it follows from \eqref{Eq-Def31int} that the equality
		\begin{equation}\label{statek}
		\begin{array}{ccccc}
		\dis \int_{\Omega}(-\Delta_D)^{\frac{s}{2}}y_{k} (-\Delta_D)^{\frac{s}{2}}\phi\,\dx&=\dis\int_{\Omega} F(x,y_{k},v_k) \phi \,\dx,
		\end{array}
		\end{equation}
		holds, for every $\phi \in \V$.
		 Moreover, it follows from \eqref{estc01} that there is a  constant $M_1>0$ independent of $k$ such that
		\begin{equation}\label{bound21}
			\|y_{k}\|_{\V}+\|y_{k}\|_{\mathcal{C}^{0,\sigma}(\bar{\Omega})}\leq M_1.
		\end{equation}
		From \eqref{bound21} and thanks to the compact embedding $\mathcal{C}^{0,\sigma}(\bar{\Omega})\hookrightarrow \mathcal{C}(\bar{\Omega})$, there exists  $y\in \V\cap \mathcal{C}^{0,\sigma}(\bar{\Omega}) $ such that (up to a subsequence if necessary), as $k\to \infty$, 
		\begin{equation}\label{c11}
			y_{k}\rightharpoonup  \bar y  \text{ weakly in }  \V
		\end{equation}
		and
		\begin{equation}\label{c21}
				y_{k}\to  \bar y  \text{ strongly in } \mathcal{C}(\bar{\Omega}).
		\end{equation}
We also have that
		\begin{equation}\label{bound22}
		\|\bar y\|_{\V}+\|\bar y\|_{\mathcal{C}(\bar{\Omega})}\leq M_1.
		\end{equation}
From \eqref{as6} and using \eqref{bound21}, there exists $C_{L,M_{1}}>0$ such that
\begin{equation}\label{boundL}
|L(x,y_{k}(x),v_{k}(x))|\leq C_{L,M_{1}}:=M_2 \;\;\text{for a.e.}\;\;x\in \Omega \mbox{ and for all } k\ge 1. 
\end{equation}
Applying \eqref{imp} with ${\tilde{p}}=+\infty$, we can deduce the existence of a constant $M_3>0$ such that
\begin{equation}\label{boundF}
|F(x,y_{k}(x),v_{k}(x))|\leq M_3\;\;\text{for a.e.}\;\;x\in \Omega \mbox{ and for all } k\ge 1. 
\end{equation}
Hence, up to a subsequence,  as $k\to \infty$, $L_{k}:=L(\cdot,y_{k}(\cdot),v_{k}(\cdot))$  converges weakly-$\star$ to $\bar{L}\in L^\infty(\Omega)$ and $F_{k}:=F(\cdot,y_{k}(\cdot),v_{k}(\cdot))$ converges weakly-$\star$ to $\bar{F}\in L^\infty(\Omega)$. This latter convergence implies also the weak convergence to $\bar{F}\in L^2(\Omega)$,  as $k\to\infty$.
Therefore,
\begin{equation}\label{co1}
\lim_{k\to \infty}\int_{\Omega}F_{k}(x)\varphi(x)\,\dx=\int_{\Omega}\bar{F}(x)\varphi(x)\,\dx \;\;\text{for all}\;\;\varphi\in L^2(\Omega)
\end{equation}
and
\begin{equation}\label{co2}
\lim_{k\to \infty}\int_{\Omega}L_{k}(x)\psi(x)\,\dx=\int_{\Omega}\bar{L}(x)\psi(x)\,\dx \;\;\text{for all}\;\;\psi\in L^2(\Omega).
\end{equation}
We define the set-valued map $Q:\Omega \times [-M_1,M_1]\to \R^2$ by
\begin{equation}\label{setQ}
Q(x,y):=\{(z,t)|M_2\geq z\geq L(x, y,u), \; t=F(x,y,u)\;\;\text{for some}\;\; u\in [\alpha,\beta]\},
\end{equation}
where $M_2$ is given in \eqref{boundL}.
Then, $Q(x,t)$ is nonempty, compact and convex in $\R^2$. This follows from \eqref{boundL}, the compactness of $[\alpha,\beta]$, the continuity of $F$, \eqref{as4} and \eqref{as9}. In addition, the set-valued map $Q$ is a Carath\'eodory multifunction (see \cite{kien2022}).
On the other hand we also introduce the set-valued map $\bar{Q}:\Omega \to \R^2$ defined by 
\begin{equation}\label{setQ0}
\bar Q(x):=\{(z,t)|M_2\geq z\geq L(x,\bar y(x), u), \; t=F(x,\bar y(x),u)\;\;\text{for some}\;\; u\in [\alpha,\beta]\},
\end{equation} 
where $M_2$ is given in \eqref{boundL}.
Then,  one can show as in \cite{kien2022} that for a.e.  $x\in \Omega$, $(\bar{L}(x),\bar{F}(x))\in Q(x)$. Therefore, for a.e.  $x\in \Omega$, there exists $u\in [\alpha,\beta]$ such that 
\begin{equation}\label{imp1}
	M_2\geq \bar{L}(x)\geq L(x,\bar y(x),u), \; \bar{F}(x)=F(x,\bar y(x),u).
\end{equation}
 Let $\mathcal{P}(\R)$ be the set of all subsets of $\R$. Define the multi-functions
$$\mathcal{H}_1:\Omega\to \mathcal{P}(\R),\quad \mathcal{H}_1(x)=(-\infty,\bar{L}(x)]\times \{\bar{F}(x)\}$$
and
$$\mathcal{H}_2:\Omega\to \mathcal{P}([\alpha,\beta]),\quad \mathcal{H}_2(x) =\Big\{u\in [\alpha,\beta]: (L(x,\bar y(x),u),F(x,\bar y(x),u))\in \mathcal{H}_1(x)\Big\}.$$
Since $L$ and $F$ are continuous with respect to the last component, we have that for a.e $x\in \Omega$, $\mathcal{H}_2(x)$ has a closed image. On the other hand,  since $\bar{L}$ and $\bar{F}$ are measurable functions, we can deduce that $\mathcal{ H}_2$ is a measurable set-valued map. Using the measurable selection theorem \cite[Theorem 8.1.3]{aubin2009}, we can deduce that there exists a measurable function $\bar u$, such that $\bar u(x)\in \mathcal{H}_2(x)$ for a.e. $x\in \Omega$. Hence, $\bar u(x)\in [\alpha,\beta]$ for a.e. $x\in \Omega$  and 
$$\bar{L}(x)\geq L(x,\bar y(x),\bar u(x)),\; \bar{F}(x)=F(x,\bar y(x),\bar u(x)).$$
So, as $k\to \infty$, 
\begin{equation}\label{conv1}
L(\cdot,y_{k}(\cdot),v_{k}(\cdot))\rightharpoonup  \bar{L}(\cdot)\geq L(\cdot,\bar y(\cdot),\bar u(\cdot))  \text{ weakly in }  L^2(\Omega)
\end{equation}
 and 
 \begin{equation}\label{conv2}
F(\cdot,y_{k}(\cdot),v_{k}(\cdot))\rightharpoonup  F(x,\bar y(\cdot),\bar u(\cdot)) \text{ weakly in }  L^2(\Omega).
 \end{equation}
Taking the limit,  as $k\to \infty$,  in \eqref{statek}, while using \eqref{c11} and \eqref{conv2} we obtain that
	\begin{equation}\label{state}
\begin{array}{ccccc}
\dis \int_{\Omega}(-\Delta_D)^{\frac{s}{2}}\bar y (-\Delta_D)^{\frac{s}{2}}\phi\,\dx&=\dis\int_{\Omega} F(x,\bar y,\bar u) \phi \,\dx\quad\forall \phi\in \V.
\end{array}
\end{equation}
On the other hand,  using \eqref{conv1} we can deduce that 
		$$\dis j=\lim_{k\to \infty}J(v_k)=\liminf_{k\to \infty}\int_{\Omega} L(x,y_{k}(x),v_{k}(x))\, \dx =\int_{\Omega}\bar{L}(x)\,\dx\geq \int_{\Omega} L(x,\bar y(x),\bar u(x))\, \dx= J(\bar u).$$
Therefore, $(\bar u,\bar y)$ is an optimal solution to the control problem \eqref{opt}-\eqref{model}.
		This completes the proof.  		
	\end{proof}
	
	We conclude this section with the following observation.	
	\begin{remark}
		In Theorem \ref{existcontrol1} we only proved the existence of optimal solutions.  Since the functional $J$ is non-convex, in general we cannot expect a unique solution to the minimization problem \eqref{opt}-\eqref{model}.
	\end{remark}
	

	\section{First order necessary optimality conditions}\label{sec-5}

	Let us set $\mathbb{ Y}:= \V\cap \mathcal{C}^{0,\sigma}(\bar{\Omega})$ and  define the control-to-state mapping
		\begin{equation}\label{ctso}
			G:L^\infty(\Omega)\to \mathbb{ Y},\;\; u\mapsto G(u)=y
		\end{equation}
		which associates to each $u\in L^\infty(\Omega)$ the unique weak solution  $y$ of \eqref{model}. \par
	In the remainder of the paper, we assume that $s>\frac{1}{2}$, $\tilde{p}>\frac{N}{2s-1}$ and $\Omega$ is of class $\mathcal{C}^{1,\sigma}$ for $0<\sigma:=2s-\frac{N}{{\tilde{p}}}-1<1$.	
	
	The aim of this section is to derive the first order necessary optimality conditions for the conrol problem  \eqref{opt}-\eqref{model} and to characterize the optimal control. But before going further, we need some regularity results for the control-to-state operator $G$ given in \eqref{ctso}. 
	Let us introduce the vector space 
	\begin{equation}\label{space1}
	\mathbb{ X}:=\Big\{y\in \mathbb{ Y}:(-\Delta_D)^sy\in L^{\tilde{p}}(\Omega)\Big\}.
	\end{equation}
	Then,  $\mathbb{ X}$ endowed with the graph norm
		\begin{equation}\label{norm1}
\|y\|_{\mathbb{ X}}:=\|y\|_{\mathbb{ Y}}+\|(-\Delta_D)^sy\|_{L^{\tilde{p}}(\Omega)},\quad\forall y\in \mathbb{ X}
		\end{equation}
	is a Banach space. Next, let us define the mapping
	\begin{equation}\label{defG}
			\mathcal G:\mathbb{X}\times L^\infty(\Omega)\to  L^{\tilde{p}}(\Omega),\qquad (y,u)\mapsto\mathcal{G}(y,u):= (-\Delta_D)^s y-F(x,y,u).
	\end{equation}
	Then, $\mathcal G$ is well defined. Indeed, from \eqref{imp}, we have that $F(\cdot,y(\cdot),u(\cdot))\in L^{\tilde{p}}(\Omega)$ for all $(y,u)\in \mathbb{ X}\times L^\infty(\Omega)$. In addition, the state equation \eqref{model}  can be viewed as  $\mathcal{G}(y,u)=0$.
	We have the following result.
	
	\begin{lemma}\label{propd}
		The mapping $\mathcal{G}$ defined in \eqref{defG} is of class $\mathcal{C}^{2}$.
	\end{lemma}
	
	
	\begin{proof}
		The first term of $\mathcal{G}$ is  linear and from the definition of the norm \eqref{norm1},  it is continuous  from $\mathbb{X}$ to $L^{\tilde{p}}(\Omega)$. Therefore,  it is of class $\mathcal{C}^{2}$. Assumption \ref{ass1} gives that the second term is of class $\mathcal C^2$ from $\mathbb{X}$ to $L^{\tilde{p}}(\Omega)$. This completes the proof.
	\end{proof}
	
	
	\begin{lemma}\label{lemmeG}
		The  mapping $G$ given in \eqref{ctso} is of class $\mathcal{C}^{2}$. In addition, under Assumption \ref{ass1}, the first and second directional derivatives of $G$ are given by $G'(u)v=z$ and $\rho=G''(u)(v,w)$
		where $u,v,w\in L^\infty(\Omega)$ and  $z,\rho\in \V$  are the unique weak solutions of 
		\begin{equation}\label{diff1}
		\left.
		\begin{array}{lllll}
		(-\Delta_D)^sz &=&\dis \frac{\partial F}{\partial y}(x,y,u)z+ \frac{\partial F}{\partial u}(x,y,u)v&\hbox{in} & \Omega,
		\end{array}
		\right.
		\end{equation}	
		and 
		\begin{align}\label{diff2}
		(-\Delta_D)^s\rho =&\dis \frac{\partial F}{\partial y}(x,y,u)\rho+\frac{\partial^2 F}{\partial y^2}(x,y,u)G'(u)vG'(u)w\notag\\
		&+\dis  \frac{\partial^2 F}{\partial y\partial u}(x,y,u)(wG'(u)v+vG'(u)w)+\frac{\partial^2 F}{\partial u^2}(x,y,u)vw \;\hbox{ in }  \Omega, 
		\end{align}	
		respectively.
		Moreover, for every $u\in L^\infty(\Omega)$, the linear mapping $v\mapsto G'(u)v$ can be extended to a linear continuous mapping from $L^2(\Omega)\to \V$. More precisely,  letting $M=\|y\|_{\mathcal{C}(\bar{\Omega})}$, there are two constants $C_{F,M}>0$ and $C=C(N,s,\Omega)>0$ such that
		\begin{equation}\label{est8}
		\|z\|_{\V}=\|G'(u)v\|_{\V}\leq C_{F,M} C(N,s,\Omega)\|v\|_{L^2(\Omega)}.
		\end{equation}	
	\end{lemma}
	
	
	\begin{proof}
		Let $u\in L^\infty(\Omega)$. It follows from Lemma \ref{propd} that $\mathcal{G}$ defined in \eqref{defG} is of class $\mathcal{C}^{2}$. Moreover,
		$$\partial_y\mathcal{G}(y,u)\varphi=(-\Delta_D)^s \varphi-\partial_yF(x,y,u)\varphi.$$
		For any $v\in L^{\tilde{p}}(\Omega)$, one can show as in Theorem \ref{theoexistence} that the problem 
		$$(-\Delta_D)^s \varphi-\partial_yF(x,y,u)\varphi=v\;\; \text{in}\;\;\Omega$$
		has a unique weak solution $\varphi$ in $\mathbb{ X}$ which depends continuously on $v$. Hence,
		$\dis \partial_y\mathcal{G}(y,u)$ defines an isomorphism from $\mathbb{X}$ to $L^{\tilde{p}}(\Omega)$. Using the Implicit Function Theorem, we can deduce that $\mathcal{G}(y,u)=(0,0)$ has a unique solution $y=G(u)$. Moreover,  the  operator $G:u\mapsto y$ is itself of class  $\mathcal{C}^{2}$. Therefore, \eqref{diff1} and \eqref{diff2} follow easily. Note that if $v\in L^2(\Omega)$, then \eqref{diff1} still has a unique weak solution $z\in \V$. Let us show the estimate \eqref{est8}. If we multiply  \eqref{diff1} with $z$ and we integrate over $\Omega$, we obtain that
		$$ \|z\|^2_{\mathbb{H}^s(\Omega)}=\dis\int_{\Omega} \frac{\partial F}{\partial y}(x,y,u)z^2\,\dx+ \int_{\Omega} \frac{\partial F}{\partial u}(x,y,u)vz\,\dx.$$
		Using \eqref{as1}, \eqref{as2} and the continuous embedding $\V\hookrightarrow L^2(\Omega)$,  we can deduce that
	$$\|z\|^2_{\V}\leq C_{F,M} C(N,s,\Omega)\|z\|_{\V}\|v\|_{L^2(\Omega)},$$
		where $M=\|y\|_{\mathcal{C}(\bar{\Omega})}$. This completes the proof.
	\end{proof}
	
Next, let us introduce the adjoint state $q\in \V$ as the unique weak solution of the adjoint equation
\begin{equation}\label{adjoint}
(-\Delta_D)^sq=\dis \frac{\partial F}{\partial y}(x,y,u)q+ \frac{\partial L}{\partial y}(x,y,u) \hbox{ in }  \Omega.
\end{equation}	

\begin{proposition}[\bf Existence of solutions to the adjoint equation]
Let $s>\frac{1}{2}$, $\tilde p>\frac{N}{2s-1}$ and $u\in L^\infty(\Omega)$. Let Assumptions \ref{ass1} and \ref{ass2} hold. If $\Omega$ is of class $\mathcal{C}^{1,\sigma}$ for $0<\sigma:=2s-\frac{N}{{\tilde{p}}}-1<1$, then there exists a unique weak solution $q\in \V\cap \mathcal{C}^{0,\sigma}(\bar{\Omega})$ to \eqref{adjoint}. Moreover, there exist two constants $C=C(N,s,\Omega)>0$ and $C_{L,M}>0$, with $M=C\left(\|F(\cdot,0,0)\|_{L^{\tilde{p}}(\Omega)}+1\right)$ such that 
\begin{equation}\label{estadjoint}
\|q\|_{\V}+\|q\|_{\mathcal{C}^{0,\sigma}(\bar{\Omega})}\leq C(N,s,\Omega)C_{L,M}.
\end{equation}
\end{proposition}

\begin{proof} From \eqref{as7}, we can deduce that there exists $C_{L,M}>0$ with $M=C\left(\|F(\cdot,0,0)\|_{L^{\tilde{p}}(\Omega)}+1\right)$ such that $$\dis \left|\frac{\partial L}{\partial y}(x,y,u)\right|\leq C_{L,M}.$$ Therefore, $\dis \frac{\partial L}{\partial u}(\cdot,y,u)\in L^{\tilde{p}}(\Omega)$. Using the same arguments as in Theorems \ref{theoexistence} and \ref{theoregularity}, we can deduce the existence of a unique weak solution $q\in \V\cap \mathcal{C}^{0,\sigma}(\bar{\Omega})$ to \eqref{adjoint}. Now, if we multiply  \eqref{adjoint} with $q$ and we integrate over $\Omega$, we get using \eqref{as1} and the continuous embedding $\V\hookrightarrow L^2(\Omega)$ that
	$$\|q\|_{\V}\leq C(N,s,\Omega)C_{L,M}.$$
Using \eqref{est2}, we get \eqref{estadjoint}. This completes the proof.	
\end{proof}

\begin{remark}
As in Lemma \ref{lemmeG}, one can prove using Assumptions  \ref{ass1} and \ref{ass2} that the mapping $u\mapsto q$ is of class $\mathcal{C}^1$, where $q$ is the weak solution to the adjoint state equation \eqref{adjoint}.
\end{remark}

	
	\begin{proposition}[\bf Twice Fr\'echet differentiability of $J$]\label{diff4}
		Let $u\in L^\infty(\Omega)$ and $y$ be the weak solution of \eqref{model}. Let Assumptions \ref{ass1} and \ref{ass2} hold. Under the hypothesis of Lemma \ref{lemmeG}, the functional $J:L^\infty(\Omega)\to \R$ defined in \eqref{opt} is twice continuously Fr\'echet differentiable and for every $v,w\in L^\infty(\Omega)$, we have that
		\begin{equation}\label{diff5}
			J'(u)v=\int_{\Omega}\left( \frac{\partial L}{\partial u}(x,y,u)+ \frac{\partial F}{\partial u}(x,y,u)q\right)v\,\dx,
		\end{equation}
		and 
		\begin{equation}\label{diff6}
		\left.
		\begin{array}{lllll}
		J''(u)[v,w] &=&\!\dis \int_{\Omega} \frac{\partial^2 L}{\partial y^2}(x,y,u)G'(u)vG'(u)w\,\dx + \int_{\Omega} \frac{\partial^2 L}{\partial y\partial u}(x,y,u)(wG'(u)v+vG'(u)w)\,\dx\\
		& &\dis +\int_{\Omega} \frac{\partial^2 L}{\partial u^2}(x,y,u)vw\,\dx+ \int_{\Omega}   \frac{\partial^2 F}{\partial y\partial u}(x,y,u)(wG'(u)v+vG'(u)w)q\,\dx\\
		&&+\dis \int_{\Omega}\left(\frac{\partial^2 F}{\partial y^2}(x,y,u)G'(u)vG'(u)w+\frac{\partial^2 F}{\partial u^2}(x,y,u)vw\right)q\,\dx,
		\end{array}
		\right.
		\end{equation}	
		where $q$ is the unique weak solution of the adjoint equation \eqref{adjoint}.
	\end{proposition}
	
	
	\begin{proof}
		Firstly, we have that  $J$ is twice continuously Fr\'echet differentiable, since by Lemma \ref{lemmeG}, $G$ has this property.\par 		
Secondly,  let $u,v,w\in L^\infty(\Omega)$.  After some straightforward calculations, we get
		\begin{equation}\label{e1}
	J'(u)v=\int_{\Omega} \frac{\partial L}{\partial u}(x,y,u)v\,\dx+ \int_{\Omega}\frac{\partial L}{\partial y}(x,y,u)G'(u)v\,\dx.
	\end{equation}
		Now, if we multiply  \eqref{diff1} with $q$ weak solution to the adjoint state \eqref{adjoint} and we integrate by parts over $\Omega$, we arrive to
		\begin{equation}\label{e2}
	 \int_{\Omega}\frac{\partial L}{\partial y}(x,y,u)G'(u)v\,\dx=\int_{\Omega}\frac{\partial F}{\partial u}(x,y,u)qv\,\dx.
		\end{equation}
Combining \eqref{e1}-\eqref{e2} leads us to \eqref{diff5}. On the other hand, after some calculations, we obtain that
	\begin{equation}\label{e3}
\left.
\begin{array}{lllll}
J''(u)[v,w] &=&\!\dis \int_{\Omega} \frac{\partial^2 L}{\partial y^2}(x,y,u)G'(u)vG'(u)w\,\dx + \int_{\Omega} \frac{\partial^2 L}{\partial y\partial u}(x,y,u)(wG'(u)v+vG'(u)w)\,\dx\\
& &\dis +\int_{\Omega} \frac{\partial^2 L}{\partial u^2}(x,y,u)vw\,\dx+\int_{\Omega} \frac{\partial L}{\partial y}(x,y,u)G''(u)[v,w]\,\dx.
\end{array}
\right.
\end{equation}	
Multiplying  \eqref{diff2} with $q$ weak solution of  \eqref{adjoint} and after an integration by parts over $\Omega$, we get 
	\begin{equation}\label{e4}
\left.
\begin{array}{lllll}
\dis \int_{\Omega} \frac{\partial L}{\partial y}(x,y,u)G''(u)[v,w]\,\dx&=& \dis \int_{\Omega}\left(\frac{\partial^2 F}{\partial y^2}(x,y,u)G'(u)vG'(u)w+\frac{\partial^2 F}{\partial u^2}(x,y,u)vw\right)q\,\dx\\
&&+\dis \int_{\Omega}   \frac{\partial^2 F}{\partial y\partial u}(x,y,u)(wG'(u)v+vG'(u)w)q\,\dx.
\end{array}
\right.
\end{equation}	
Combining \eqref{e3}-\eqref{e4}, we deduce \eqref{diff6}. The proof is finished.
	\end{proof}
	
	
We introduce the following notion of local  solutions.
	
	\begin{definition}\label{defopt}
	Let $1\le p\le \infty$. 
		We say that $u\in \U$ is an $L^p$-local solution of \eqref{opt} if there exists $\varepsilon>0$ such that 
\begin{equation}\label{pls}		
	J(u)\leq J(v) \mbox{ for every }v\in \U\cap B_\varepsilon^p(u),
\end{equation}	
	where $B_\varepsilon^p(u):=\{u\in L^p(\Omega):\|v-u\|_{L^p(\Omega)}\leq \varepsilon\}$. We say that $u$ is a strict local minimum of \eqref{opt} if the  inequality \eqref{pls} is strict whenever $v\neq u$.
	\end{definition}
	
Let us introduce the Hamiltonian $H$ of \eqref{opt}-\eqref{model} given by
\begin{equation}\label{ham}
H:\Omega\times \R\times\R\times\R\to \R, \;\;H(x,t,\eta,\xi)=L(x,t,\xi)+\eta F(x,t,\xi).
\end{equation}
Note that from Assumptions \ref{ass1} and \ref{ass2}, $H$ is of class $\mathcal{C}^2$ with respect to the last component.\par
Before going further, we establish the following important result.

\begin{proposition}\label{convergence}
	Assume that $s>\frac{1}{2}$, $\frac{N}{2s-1}<{\tilde{p}}$ and $\Omega$ is of class $\mathcal{C}^{1,\sigma}$ for $0<\sigma:=2s-\frac{N}{{\tilde{p}}}-1<1$.  Consider a sequence $(u_k)_{k\geq1}$ bounded in $L^\infty(\Omega)$ converging strongly to $u$ in  $L^{\tilde{p}}(\Omega)$, as $k\to\infty$. If we denote by $y_k:=y(u_k)$ and $y_u:=y(u)$ the states associated to $u_k$ and to $u$, respectively, then 
	\begin{equation}\label{conv0}
	\dis \lim_{k\to \infty}\left(\|y_{k}-y_u\|_{\V}+\|y_{k}-y_u\|_{\mathcal{C}^{0,\sigma}(\bar{\Omega})}\right)=0.
	\end{equation}
\end{proposition}

\begin{proof}
	Let $(u_k)_{k\geq1}$ be a  bounded sequence  in $L^\infty(\Omega)$ such that, as $k\to\infty$, 
	\begin{equation}\label{l0}
	u_k  \to   u  \text{ strongly in }  L^{\tilde{p}}(\Omega). 
	\end{equation}
Using \eqref{est0}, we can deduce that there is a  constant $M>0$ such that for all $k\ge 1$, 
	\begin{equation}\label{est02}
	\|y_{k}\|_{\V}+\|y_{k}\|_{L^\infty(\Omega)}\leq M.
	\end{equation}
	Moreover, $y_{k}-y_{u}$ satisfies
	\begin{equation}\label{1}
	\int_{\Omega}(-\Delta_D)^{\frac{s}{2}}(y_{k}-y_{u}) (-\Delta_D)^{\frac{s}{2}}\phi\,\dx=\dis\int_{\Omega} (F(x,y_{k},u_k)-F(x,y_{u},u)) \phi \,\dx\;\;\;\;\forall \phi\in \V.
	\end{equation}
	Taking $\phi=y_{k}-y_{u}$ as a test function  in \eqref{1} and thanks to the monotonicity of $F$ and Assumption \ref{ass1}, we get
	\begin{equation*}
	\begin{array}{llll}
	\dis \|y_{k}-y_{u}\|^2_{\V}&=& \dis\int_{\Omega} (F(x,y_{k},u_k)-F(x,y_{u},u))(y_{k}-y_{u})\,\dx\\
	&=& \dis \dis\int_{\Omega} (F(x,y_{k},u_k)-F(x,y_{u},u_k))(y_{k}-y_{u})\,\dx\\
	&&+\dis\int_{\Omega} (F(x,y_{u},u_k)-F(x,y_{u},u))(y_{k}-y_{u})\,\dx\\
	&\leq& \dis\int_{\Omega} (F(x,y_{u},u_k)-F(x,y_{u},u))(y_{k}-y_{u})\,\dx\\
	&\leq& \dis C_{F,M}\int_{\Omega} |u_k-u||y_{k}-y_{u}|\,\dx\\
	&\leq& \dis C_{F,M}\|u_k-u\|_{L^p(\Omega)}\|y_{k}-y_{u}\|_{L^{p'}(\Omega)},
	\end{array}
	\end{equation*}
	where $p'=\frac{{\tilde{p}}}{{\tilde{p}}-1}$. Since ${\tilde{p}}>\frac{N}{2s-1}$,  we have that $p'<\frac{N}{N-2s+1}<\frac{N}{N-2s}$. Therefore, if $N>2s$, we  get  from \eqref{equality0} that $\V\hookrightarrow L^{\frac{2N}{N-2s}}(\Omega)\hookrightarrow L^{p'}(\Omega)$. Also if $N\leq 2s$, it follows from \eqref{equality0} that $\V\hookrightarrow L^{p'}(\Omega)$. Therefore, 
	\begin{equation*}
	\dis \|y_{k}-y_{u}\|^2_{\V}\leq \dis C_{F,M}C(p,\Omega)\|u_k-u\|_{L^p(\Omega)}\|y_{k}-y_{u}\|_{\V},
	\end{equation*}
	and so
	\begin{equation}\label{l1}
	\dis \|y_{k}-y_{u}\|_{\V}\leq \dis C_{F,M}C(p,\Omega)\|u_k-u\|_{L^p(\Omega)}.
	\end{equation}
	Taking the limit,  as $k\to \infty$,   in \eqref{l1} while using \eqref{l0} yields 
	the strong convergence of the sequence $(y_{k})_k$ to $y_u$ in $\V$.  Now, \eqref{estc00} along with the boundedness of the sequence $(u_k)_{k\geq1}$ in $L^\infty(\Omega)$ imply that the sequence $(y_{k})_{k\geq1}$ is bounded in $\mathcal{C}^{0,\sigma}(\bar{\Omega})$. Therefore, using the compactness of the embedding $\mathcal{C}^{0,\sigma}(\bar{\Omega})\hookrightarrow \mathcal{C}(\bar{\Omega})$ and the latter convergence, we arrive to $y_{k}\to y_u$ strongly in $\mathcal{C}(\bar{\Omega})$, as $k\to\infty$. Hence,  as $k\to\infty$,  we have that
	\begin{equation}\label{strong}
	y_{k}\to y_u \;\;\;\text{strongly in}\;\;\;L^{\tilde{p}}(\Omega).
	\end{equation} 
	On the other hand, recalling again that $y_{k}-y_{u}$ satisfies
	\begin{equation}\label{2}
	(-\Delta_D)^{s}(y_{k}-y_{u})=F(x,y_{k},u_k)-F(x,y_{u},u),
	\end{equation}
	we can deduce from \eqref{est2} that
	\begin{equation}\label{est6}
	\|y_{k}-y_{u}\|_{\mathcal{C}^{0,\sigma}(\bar\Omega)}\leq C\left(\|y_{k}-y_{u}\|_{\V}+\|F(\cdot,y_{k}(\cdot),u_k(\cdot))-F(\cdot,y_u(\cdot),u(\cdot))\|_{L^{\tilde{p}}(\Omega)}\right).
	\end{equation} 
	Now, using \eqref{as2} and \eqref{est02}, we obtain that 
	\begin{equation}\label{est7}
	\|F(\cdot,y_{k}(\cdot),u_k(\cdot))-F(\cdot,y_u(\cdot),u(\cdot))\|_{L^{\tilde{p}}(\Omega)}\leq C_{F,M}\left(\|y_{k}-y_u\|_{L^{\tilde{p}}(\Omega)}+\|u_k-u\|_{L^{\tilde{p}}(\Omega)}\right).
	\end{equation} 
	Combining \eqref{l1}, \eqref{est6}, \eqref{est7}, \eqref{l0} and letting $k\to \infty$ in \eqref{l1} and \eqref{est6},  while using \eqref{strong}, we get \eqref{conv0}. The proof is finished.
\end{proof}

	The following result is crucial for the rest of the paper.
	
	\begin{theorem}[\bf First order necessary optimality conditions]\label{theoSO}
	Assume that Assumptions \ref{ass1} and \ref{ass2} hold. 
		Let $u\in\mathcal U$ be an $L^p$-local minimum ($1\leq p\leq \infty $) for  \eqref{opt}-\eqref{model}. Then,
		\begin{equation}\label{ineq}
			J'(u)(v-u)\geq 0\;\;\;\text{for every}\;\;\; v\in \U,
		\end{equation}
	equivalently	
		\begin{equation}\label{ineq1}
		\int_{\Omega}\frac{\partial H}{\partial u}(x,y,q,u)(v-u)\,\dx\geq 0\;\;\;\text{for every}\;\;\; v\in \U,
		\end{equation}
		where $q$ is the unique weak solution of \eqref{adjoint} and $H$ is the Hamiltonian given in \eqref{ham}. Moreover, if $1\le p<\infty$, then
		\begin{equation}\label{min}
		\dis \int_{\Omega} H\left(x,y(x),q(x),u(x)\right)\,\dx=\min_{v\in\mathcal U}\int_{\Omega} H\left(x,y(x),q(x),v(x)\right)\,\dx.
		\end{equation} 
	\end{theorem}
	
	
	\begin{proof}
		Let $v\in \U$ be arbitrary. Since $\U$ is convex,  we have that $u+\lambda(v-u)\in \U$ for all $\lambda\in (0,1]$.  Since $u$ is an $L^p$-local minimum,  it follows that $J(u+\lambda(v-u))\geq J(u)$. Hence,
		$$\dis\frac{J(u+\lambda(v-u))-J(u)}{\lambda}\geq 0 \mbox{ for all } \lambda\in (0,1].$$
		Letting $\lambda\downarrow 0$ in the latter inequality, we get 	\eqref{ineq}.  This implies that
	\begin{equation}\label{SRW}
	\dis\int_{\Omega}\left( \frac{\partial L}{\partial u}(x,y,u)+ \frac{\partial F}{\partial u}(x,y,u)q\right)(v-u)\,\dx\geq 0\;\;\;\text{for every}\;\;\; v\in \U.
	\end{equation}
	Using the form of $H$ given in \eqref{ham},  then \eqref{ineq1} follows from \eqref{SRW}. \par
Let us prove \eqref{min}. The proof is inspired from the local case given in \cite[Theorem 4.1]{casas2020o}.  Assume that $1\le p<\infty$. Let $v\in \mathcal{U}$ be a fixed control. We define the function $h(x):=F(x,y(x),v(x))-F(x,y(x),u(x))$ for a.e. $x\in \Omega$. Then, from \eqref{imp}, we deduce that $h\in L^{\tilde{p}}(\Omega)$. Next,  let $(v_j)_{j\geq1}$ be a dense sequence in $L^1(\Omega)$. For every $k\geq 1$, we define the function $g_k\in \left(L^1(\Omega)\right)^{k+1}$ by $g_k:=(1,v_1,\cdots, v_k)$. Given $\rho\in (0,1)$ arbitrarily, we deduce from Lyapunov's convexity theorem the existence of measurable sets $E^k_{\rho}\subset \Omega$ such that
\begin{equation}
\dis \int_{E^k_{\rho}}g_k(x)\,\dx=\rho\int_{\Omega}g_k(x)\,\dx \;\;\;\forall k\geq 1.
\end{equation}
Looking at the first component of the above vector identity, we have that
\begin{equation}
\dis  |E^k_{\rho}|=\rho|\Omega|\;\;\;\forall k\geq 1.
\end{equation}
Now, considering the remaining components, we observe that
$$\int_{\Omega}\frac{1}{\rho}\chi_{E^k_{\rho}}v_j\,\dx=\frac{1}{\rho}\int_{E^k_{\rho}}v_j\,\dx=\int_{\Omega}v_j(x)\,\dx\;\;\;\forall j\geq 1.$$
This implies that
\begin{equation}\label{dens}
\dis \lim_{k\to \infty}\int_{\Omega}\frac{1}{\rho}\chi_{E^k_{\rho}}v_j\,\dx=\int_{\Omega}v_j(x)\,\dx\;\;\;\forall j\geq 1.
\end{equation}
From the density of $(v_j)_{j\geq1}$ in $L^1(\Omega)$ we can deduce from \eqref{dens} that
\begin{equation}\label{k}
\dis \lim_{k\to \infty}\int_{\Omega}\frac{1}{\rho}\chi_{E^k_{\rho}}v\,\dx=\int_{\Omega}v(x)\,\dx\;\;\;\forall v\in L^1(\Omega).
\end{equation}
This means that $\dis \frac{1}{\rho}\chi_{E^k_{\rho}}\stackrel{\star}\rightharpoonup 1  \text{  in }  L^\infty(\Omega)$,  as $k\to \infty$. Since $h\in L^{\tilde{p}}(\Omega)$, it holds that $\left(1-\frac{1}{\rho}\chi_{E^k_{\rho}}\right)h\rightharpoonup 0 \text{ weakly in }  L^{\tilde{p}}(\Omega)$, as $k\to\infty$.  Since the embedding $\V\hookrightarrow L^p(\Omega)$ is compact for every $p\in [1,\frac{2N}{N-2s})$, we have the compact embedding $ L^{p'}(\Omega)\hookrightarrow\V^\star$ for any $p'>\frac{N}{N+2s}$.  If $s\in(\frac{1}{2},1)$,  $\frac{N}{2s-1}>\frac{N}{N+2s}$, then the embedding $ L^{p'}(\Omega)\hookrightarrow\V^*$ is compact for every $p'>\frac{N}{2s-1}$. Consequently, $L^{\tilde{p}}(\Omega)$ is compactly embedded in $\V^*$. Thus,  we can select $E_{\rho}=E^{k}_{\rho}$ with some sufficiently large $k$, such that
\begin{equation}\label{convimp}
\left\|\left(1-\frac{1}{\rho}\chi_{E_{\rho}}\right)h\right\|_{\V^\star}<\rho.
\end{equation}
We also have from \eqref{k} that,
\begin{equation}\label{k0}
\dis \lim_{\rho\to \infty}\int_{\Omega}\frac{1}{\rho}\chi_{E_{\rho}}v\,\dx=\int_{\Omega}v(x)\,\dx\;\;\;\forall v\in L^1(\Omega).
\end{equation}
Now, we define 
\begin{equation}\label{defuro}
u_{\rho}:=\left\{\begin{array}{lll}
v(x)& &\text{ if } x\in E_{\rho},\\
u(x)&& \text{ otherwise }.
\end{array}
\right.
\end{equation}
Then, $u_{\rho}\in \mathcal{U}$. Let us denote by $y_{\rho}$ the state associated to $u_{\rho}$. We also set $z_\rho:=\frac{1}{\rho}(y_{\rho}-y)$, where $y$ is the weak solution of \eqref{model} associated to $u$. Then,  $z_{\rho}$ is the weak solution of
\begin{equation}\label{p1}
\left.
\begin{array}{lllll}
(-\Delta_D)^sz_{\rho}&=&\dis \frac{\partial F}{\partial y}(x,y+\theta_{\rho}(y_{\rho}-y),u_{\rho})z_{\rho}+ \frac{1}{\rho}\chi_{E_{\rho}}h&\hbox{in} & \Omega,
\end{array}
\right.
\end{equation}	
where $\theta_{\rho}:\Omega\to[0,1]$ is a measurable function.  Next, we consider the problem
\begin{equation}\label{p2}
\left.
\begin{array}{lllll}
(-\Delta_D)^sz&=&\dis \frac{\partial F}{\partial y}(x,y,u)z+[F(x,y,v)-F(x,y,u)]&\hbox{in} & \Omega.
\end{array}
\right.
\end{equation}	
We claim that $z_{\rho}$ weak solution of \eqref{p1} converges strongly in $\V$ to $z$ weak solution of \eqref{p2}, as $\rho\to 0$. 
First, we note that under Assumption \ref{ass1}, $z \in \V\cap \mathcal{C}^{0,\sigma}(\bar{\Omega})$. Next, we write for any $r\geq1$,
\begin{equation}\label{ine}
	\|u_{\rho}-u\|_{L^r(\Omega)}=\|u_{\rho}-u\|_{L^r(E_{\rho})}\leq 2\max(|\alpha|,|\beta|)|E_{\rho}|^{\frac{1}{r}}\leq 2\max(|\alpha|,|\beta|)(\rho|\Omega|)^{\frac{1}{r}}.
\end{equation}
Therefore,  as $\rho\to 0$, 
\begin{equation}\label{c0}
u_{\rho}\to u\;\; \text{strongly in}\;\;\; L^r(\Omega)
\end{equation}
for every $r\geq 1$. In particular,  for $r={\tilde{p}}$, we deduce from Proposition \ref{convergence} that,  as $\rho\to 0$,
\begin{equation}\label{c1}
 y_{\rho}\to y\;\; \text{strongly in}\;\;\; \V\cap \mathcal{C}^{0,\sigma}(\bar{\Omega}).
\end{equation}
On the other hand, $z_{\rho}-z$ is the weak solution of
\begin{align}\label{p3}
(-\Delta_D)^s(z_{\rho}-z)=&\dis \frac{\partial F}{\partial y}(x,y+\theta_{\rho}(y_{\rho}-y),u_{\rho})(z_{\rho}-z)+ \left(\frac{1}{\rho}\chi_{E_{\rho}}-1\right)h\notag\\
&\dis + \left[\frac{\partial F}{\partial y}(x,y+\theta_{\rho}(y_{\rho}-y),u_{\rho})-\frac{\partial F}{\partial y}(x,y,u)\right]z\;\qquad\hbox{ in } \; \Omega.
\end{align}	
Set $M:=2\|y\|_{\mathcal{C}(\bar{\Omega})}$. If we multiply  \eqref{p3} by $z_{\rho}-z$ and integrate over $\Omega$, we get
\begin{equation}\label{p4}
\left.
\begin{array}{lllll}
\|z_{\rho}-z\|^2_{\V}&=&\dis\int_{\Omega} \frac{\partial F}{\partial y}(x,y+\theta_{\rho}(y_{\rho}-y),u_{\rho})(z_{\rho}-z)^2\,\dx+ \int_{\Omega}\left(\frac{1}{\rho}\chi_{E_{\rho}}-1\right)h(z_{\rho}-z)\,\dx\\
&&\dis + \int_{\Omega} \left[\frac{\partial F}{\partial y}(x,y+\theta_{\rho}(y_{\rho}-y),u_{\rho})-\frac{\partial F}{\partial y}(x,y,u)\right]z(z_{\rho}-z)\,\dx.
\end{array}
\right.
\end{equation}	
Using \eqref{as1}, we arrive to 
\begin{equation}\label{p5}
\left.
\begin{array}{lllll}
\|z_{\rho}-z\|^2_{\V}&\leq &\dis \|z\|_{\mathcal{C}(\bar{\Omega})}\left\|\frac{\partial F}{\partial y}(x,y+\theta_{\rho}(y_{\rho}-y),u_{\rho})-\frac{\partial F}{\partial y}(x,y,u)\right\|_{L^2(\Omega)}\|z_{\rho}-z\|_{L^2(\Omega)}\\ &&+ \dis\left\|\left(\frac{1}{\rho}\chi_{E_{\rho}}-1\right)h\right\|_{\V^\star}\|z_{\rho}-z\|_{\V}.
\end{array}
\right.
\end{equation}	
Using the Mean Value Theorem and \eqref{as2} we obtain  after some calculations
\begin{equation}\label{estF}
\left\|\frac{\partial F}{\partial y}(x,y+\theta_{\rho}(y_{\rho}-y),u_{\rho})-\frac{\partial F}{\partial y}(x,y,u)\right\|_{L^2(\Omega)}\leq C_{F,M}\left[\|y_{\rho}-y\|_{L^2(\Omega)}+\|u_{\rho}-u\|_{L^2(\Omega)}\right].
\end{equation}

Combining \eqref{est5}-\eqref{estF}, we deduce from \eqref{p5} that
 \begin{equation}\label{p6}
 \left.
 \begin{array}{lllll}
 \|z_{\rho}-z\|^2_{\V}&\leq &\dis C_{F,M}C(\Omega,N,s)\|z\|_{\mathcal{C}(\bar{\Omega})}\left[\|y_{\rho}-y\|_{\V}+\|u_{\rho}-u\|_{L^2(\Omega)}\right]\|z_{\rho}-z\|_{\V}\\ &&+ \dis\left\|\left(\frac{1}{\rho}\chi_{E_{\rho}}-1\right)h\right\|_{\V^\star}\|z_{\rho}-z\|_{\V}.
 \end{array}
 \right.
 \end{equation}	
 Hence, 
  \begin{equation}\label{p7}
 \left.
 \begin{array}{lllll}
 \|z_{\rho}-z\|_{\V}&\leq &\dis C_{F,M}C(\Omega,N,s)\|z\|_{\mathcal{C}(\bar{\Omega})}\left[\|y_{\rho}-y\|_{\V}+\|u_{\rho}-u\|_{L^2(\Omega)}\right]\\ &&+ \dis\left\|\left(\frac{1}{\rho}\chi_{E_{\rho}}-1\right)h\right\|_{\V^\star}.
 \end{array}
 \right.
 \end{equation}	
 Taking the limit in \eqref{p7},  as $\rho\to 0$, while using \eqref{c0}, \eqref{c1} and \eqref{convimp}, we obtain the claim.
 
 Note also that \eqref{ine} implies that $u_{\rho}\in B_{\varepsilon}^p(u)$ for every $\rho$ sufficiently small. Therefore, from the optimality of $u$, we get
 \begin{equation}\label{limj}
\dis 0\leq \frac{J(u_{\rho})-J(u)}{\rho}=	\int_{\Omega} \frac{\partial F}{\partial y}(x,y+\theta_{\rho}(y_{\rho}-y),u_{\rho})z_{\rho}\,\dq+ \int_{\Omega} \frac{1}{\rho}\chi_{E_{\rho}}[L(x,y,v)-L(x,y,u)]\,\dx.
 \end{equation}
Using the Mean Value Theorem and \eqref{as7}, we obtain that
\begin{equation}\label{estL}
\left\|\frac{\partial L}{\partial y}(x,y+\theta_{\rho}(y_{\rho}-y),u_{\rho})-\frac{\partial L}{\partial y}(x,y,u)\right\|_{L^2(\Omega)}\leq C_{L,M}\left[\|y_{\rho}-y\|_{L^2(\Omega)}+\|u_{\rho}-u\|_{L^2(\Omega)}\right].
\end{equation}
In addition, from \eqref{as6}, we have that $L(\cdot,y,v)-L(\cdot,y,u)\in L^1(\Omega)$. Therefore,  passing to the limit,  as $\rho\to 0$,  in \eqref{limj}, while using \eqref{k0}, \eqref{c0} and \eqref{c1}, leads us to
 \begin{equation*}
\dis 0\leq \int_{\Omega} \frac{\partial F}{\partial y}(x,y,u)z\,\dq+ \int_{\Omega}[L(x,y,v)-L(x,y,u)]\,\dx,
\end{equation*}
which by \eqref{adjoint} is equivalent to 
\begin{equation*}
\dis 0\leq \int_{\Omega} \left[(-\Delta_D)^sq- \frac{\partial F}{\partial y}(x,y,u)q\right]z\,\dq+ \int_{\Omega}[L(x,y,v)-L(x,y,u)]\,\dx.
\end{equation*}
Integrating this latter inequality by parts and using \eqref{p2}, we arrive to 
\begin{equation*}
\dis 0\leq \int_{\Omega} \left[F(x,y,v)-F(x,y,u)\right]\,\dx+ \int_{\Omega}[L(x,y,v)-L(x,y,u)]\,\dx.
\end{equation*}
Hence, 
$$\dis \int_{\Omega} H(x,y(x),q(x),u(x))\,\dx\leq \int_{\Omega} H(x,y(x),q(x),v(x))\,\dx.$$
Since $v\in \mathcal{U}$ is arbitrary, we can deduce that \eqref{min} holds and the proof is finished.
	\end{proof}
	
Now, we can state and prove the pointwise Pontryagin type principle.

\begin{proposition}\label{pr1}
Under the hypothesis of Theorem \ref{theoSO} with $1\le p<\infty$,  the following pointwise Pontryagin principle holds:
	\begin{equation}\label{pont}
	\dis 	H(x,y(x),q(x),u(x))=\min_{t\in [\alpha,\beta]} H(x,y(x),q(x),t)\;\;\; \text{for a.e.}\; x\in \Omega,
	\end{equation}
	where $y$ and $q$ are the state and adjoint state associated with $u$, respectively.
\end{proposition}

\begin{proof}
Let $(\alpha_j)_{j\geq 0}$ be a sequence of rational numbers contained in $[0,1]$. For every $j$, we set $v_j:=\alpha \alpha_j+(1-\alpha_j)\beta$. Then,  every function $v_j$ belongs to $L^{\infty}(\Omega)$ and $v_j\in [\alpha,\beta]$. Hence, $v_j\in \mathcal{U}$. Now,  let us introduce the functions $F_0,F_j:\Omega\to \R$ given by
$$F_0(x)=H(x,y(x),q(x),u(x))\;\;\;\;\text{and}\;\;\;\; F_j(x)=H(x,y(x),q(x),v_j(x)), \;\;\;j\in\N.$$
We introduce the set of Lebesgue regular points $E_0$ and $(E_j)_{j\geq 1}$, which are known to satisfy $|E_j|=|\Omega|$ for $j\in\N\cup\{0\}$ and 
\begin{equation}\label{im}
\lim_{r\to 0}\frac{1}{|B_r(x_0)|}\int_{B_r(x_0)}F_j(x)\,\dx=F_j(x_0),\;\;\forall x_0\in E_j,\;\;j\in\N\cup\{0\},
\end{equation}
where $B_r(x_0)$ is the Euclidean ball in $\R^N$ of center $x_0$ and radius $r$. By setting $E:=\bigcap_{j=0}^{\infty}E_j$,  we have that $|E|=|\Omega|$ and \eqref{im} remains true for all $x_0\in E$. Let $x_0\in E$, $r>0$ and define 
\begin{equation}\label{def}
u_{j,r}:=\left\{\begin{array}{lll}
u(x)& \text{if}& x\notin B_r(x_0),\\
v_j(x)& \text{if}& x\in B_r(x_0),\;\;j\in\N\cup\{0\}.
\end{array}
\right.
\end{equation}
Then,  $u_{j,r}\in \mathcal{U}$. Hence, from \eqref{min}, we can deduce that 
$$\dis \int_{\Omega} H(x,y(x),q(x),u(x))\,\dx\leq \int_{\Omega} H(x,y(x),q(x),u_{j,r}(x))\,\dx,\;\;\forall\;j\in\N\cup\{0\}.$$
Using \eqref{def}, the preceding estimate implies that for all $j\in \N\cup\{0\}$, 
 $$\dis \frac{1}{|B_r(x_0)|}\int_{B_r(x_0)} H(x,y(x),q(x),u(x))\,\dx\leq \frac{1}{|B_r(x_0)|}\int_{B_r(x_0)}H(x,y(x),q(x),v_j(x))\,\dx.$$
 Taking the limit of the latter inequality,  as $r\to 0$, while using \eqref{im}, we arrive to
 $$\dis H(x_0,y(x_0),q(x_0),u(x_0))\leq H(x_0,y(x_0),q(x_0),v_j(x_0)),\;\;\forall\; j\in \N\cup\{0\}.$$
From the continuity of the functions $L$ and $F$ with respect to last component, we obtain that $H$ is also continuous with respect to the last component. Since $(v_j(x_0))_{j\geq1}$ is dense in $[\alpha,\beta]$, we get
\begin{equation}\label{MJ}
\dis H(x_0,y(x_0),q(x_0),u(x_0))\leq H(x_0,y(x_0),q(x_0),t),\;\;\;\forall t\in [\alpha,\beta].
\end{equation}
Since $x_0$ is arbitrary, we can deduce \eqref{pont} from \eqref{MJ}. The proof is finished.
\end{proof}

	
	\begin{lemma}\label{cont1} Let $s>\frac{1}{2}$, $\tilde p>\frac{N}{2s-1}$ and $u\in L^\infty(\Omega)$.  Let Assumptions \ref{ass1} and \ref{ass2} hold. If $\Omega$ is of class $\mathcal{C}^{1,\sigma}$ for $0<\sigma:=2s-\frac{N}{{\tilde{p}}}-1<1$, then the linear mapping $v\mapsto J'(u)v$ can be extended to a linear continuous mapping $J'(u):L^2(\Omega)\to \R$ given by \eqref{diff5}. Moreover,  the bilinear  mapping $(v,w)\mapsto {J}''(u)[v,w]$ can be extended to a bilinear continuous mapping ${J}''(u):L^2(\Omega)\times L^2(\Omega)\to \R$ given by \eqref{diff6}. In addition,  there is a constant $C=C(N,s,\Omega,\alpha,\beta)>0$ such that for all $v, w\in L^2(\Omega)$,
		\begin{equation}\label{estJp}
			|J'(u)v|\leq C\|v\|_{L^2(\Omega)}
		\end{equation}
		and 
		\begin{equation}\label{estJs}
		|J''(u)[v,w]|\leq C\|v\|_{L^2(\Omega)}\|w\|_{L^2(\Omega)}.
		\end{equation}
	\end{lemma}	
	
	
	\begin{proof} Let $u\in L^\infty(\Omega)$ and $v\in L^2(\Omega)$. From \eqref{diff5}, we have that
		$$
	J'(u)v=\int_{\Omega}\left( \frac{\partial L}{\partial u}(x,y,u)+ \frac{\partial F}{\partial u}(x,y,u)q\right)v\,\dx, 
		$$
		where  $q$ is the weak solution of \eqref{adjoint}.
		Using \eqref{est1}, \eqref{as2}, \eqref{as7} and \eqref{estadjoint},  we have that there is a constant $C>0$ independent of $v$ such that \eqref{estJp} holds.
		Thus, the mapping $v\mapsto J'(u)v$ is linear and continuous on $L^2(\Omega)$. Next,
	let $v,w\in L^2(\Omega)$. From \eqref{diff6}, we have that
			\begin{align}\label{ww}
			J''(u)[v,w] =&\!\dis \int_{\Omega} \frac{\partial^2 L}{\partial y^2}(x,y,u)G'(u)vG'(u)w\,\dx + \int_{\Omega} \frac{\partial^2 L}{\partial y\partial u}(x,y,u)(wG'(u)v+vG'(u)w)\,\dx\notag\\
			 &\dis +\int_{\Omega} \frac{\partial^2 L}{\partial u^2}(x,y,u)vw\,\dx+ \int_{\Omega}   \frac{\partial^2 F}{\partial y\partial u}(x,y,u)(wG'(u)v+vG'(u)w)q\,\dx\notag\\
			&+\dis \int_{\Omega}\left(\frac{\partial^2 F}{\partial y^2}(x,y,u)G'(u)vG'(u)w+\frac{\partial^2 F}{\partial u^2}(x,y,u)vw\right)q\,\dx.
			\end{align}
			Using Cauchy Schwarz's inequality, \eqref{as2}, \eqref{as7}, \eqref{est8} and \eqref{estadjoint}, we get	 from \eqref{ww} that there is a constant $C>0$ such that \eqref{estJs} holds.
			Hence, the mapping $(v,w)\mapsto {J}''(u)[v,w]$ is a bilinear continuous mapping on $L^2(\Omega)\times L^2(\Omega)$. This completes the proof.
	\end{proof}

	\section{Second order optimality conditions}\label{second}

The main concern of this section is to derive the second order necessary and sufficient conditions of optimalty for the control problem \eqref{opt}-\eqref{model}.

	\subsection{Second order necessary optimality conditions}\label{sufficientoptcond}
	
Since the cost functional $J$ associated to the optimization problem \eqref{opt}-\eqref{model} is non-convex, we have that  the first order optimality conditions given in Theorem \ref{theoSO} are necessary but not sufficient for optimality.  
	
We first introduce the cone of critical directions associated to a control $u\in\mathcal U$ defined by
	\begin{equation}\label{ccone}
		\mathcal C_{u}:= \Big\{v\in L^2(\Omega) : J'(u)v=0 \;\text{ and  $v$ fulfills}\; \eqref{eq1}\Big\},
	\end{equation}
	that is, for a.e $x\in \Omega$,  
	\begin{equation}\label{eq1}
		\left\{
		\begin{array}{lllll}
			\dis  v(x)\geq 0 &\mbox{ if }&  u(x)=\alpha ,\\
			\dis v(x)\leq 0&\mbox{ if} & u(x)=\beta,\\
			\dis 0 &\mbox{ if } &d(x)\neq 0,
		\end{array}
		\right.
	\end{equation}
	where 
	$$d(x):=\frac{\partial L}{\partial u}(x,y,u)+ \frac{\partial F}{\partial u}(x,y,u)q.$$
%
	In the rest of the paper, we will adopt the notation $J''(u)v^2:=J''(u)[v,v]$.
	
	\begin{theorem}[\bf Second order necessary optimality conditions]\label{theo-514}
		Let $u\in \mathcal{U}$ be an $L^p$-local solution of \eqref{opt} with $1\leq p\leq \infty$. Then,  $J''(u)v^2\geq 0$ for all $v\in \mathcal C_u$.
	\end{theorem}
	
	\begin{proof}
		The proof follows as in \cite[Theorem 5.16]{KMW2022}  (see also \cite[pp. 246]{fredi2010} for the local case).  Indeed,  Let $v\in \mathcal C_u$. We define for every $n\in \N$ and for almost every $x\in \Omega$, the function
		\begin{equation*}
		v_n(x):=
		\begin{cases}
		0\;\;&\mbox{ if }\; u(x)\in \dis \left(\alpha,\alpha+\frac{1}{n}\right)\bigcup \left(\beta-\frac{1}{n},\beta\right),\\
		\dis  \Pi_{[-n,n]}(v(x)) &\text{ otherwise }.
		\end{cases}
		\end{equation*}
		Then, $\bar{u}:=u+\theta v_n\in \U$ for $\theta\in (0,1)$. Moreover,  using that $u$ is a local optimal control, we deduce that 
		\begin{equation}\label{JJ}
			0\leq \frac{{J}(\bar{u})-{J}(u)}{\theta}={J}'(u)v_n+\frac{1}{2}\theta{J}''(u)v_n^2+\theta^{-1}r^2(u,\theta v_n),
		\end{equation}
		where $r^2(u,\theta v_n)$ represents the second-order remainder. Since $v\in \mathcal C_u$, it follows that $v_n\in \mathcal C_u$.  Thus,  $\dis {J}'(u)v_n=0$. Therefore,
		dividing both sides of \eqref{JJ} by $\theta$, we get
				\begin{equation}\label{Jj}
			0\leq \frac{1}{2}{J}''(u)v_n^2+\theta ^{-2}r^2(u,\theta v_n).
		\end{equation}
		Taking the limit in \eqref{Jj},  as $\theta\to 0$,  yields
		\begin{equation}\label{1-1}
		{J}''(u)v_n^2\geq 0.
		\end{equation} 
		It remains to prove that,  as $n\to \infty$, $v_n\to v$ in $L^2(\Omega)$. First,  we note that for a.e $x\in \Omega$, $v_n(x)\to v(x)$,  as $n\to \infty$. In addition,  $|v_n(x)|^2\leq |v(x)|^2$ for a.e. $x\in\Omega$ and  for all $n\in \N$.  Using the Lebesgue dominated convergence theorem, we can deduce that,  as $n\to \infty$, $v_n\to v$ in $L^2(\Omega)$. Hence,  taking the limit,  as $n\to \infty$,  in \eqref{1-1} and using the continuity of $v\mapsto {J}''(u)v^2$ in $L^2(\Omega)$, we obtain that ${J}''(u)v^2\geq 0$. The proof is finished.
	\end{proof}
	
\subsection{Second order sufficient optimality conditions for $L^{\infty}$-local solutions}
In this section, we impose as in the local case \cite{casas2020o} the following suitable structural assumption on $\frac{\partial H}{\partial u}(x,y,q,u)$, where $u\in \mathcal{U}$ is a control satisfying \eqref{ineq1}, and $y$, $q$ are the associated state and adjoint state, respectively.

	\begin{assumption}\label{structural}
		There exist three constants $K>0$, $\varepsilon_0>0$ and $\gamma\in (0,+\infty]$ such that
		\begin{equation}\label{struct}
		\left|\left\{x\in \Omega: \left|\frac{\partial L}{\partial u}(x,y(x),u(x))+ q(x) \frac{\partial F}{\partial u}(x,y(x),u(x))\right|<\varepsilon \right\}\right|\leq K\varepsilon^{\gamma}, \quad\forall \varepsilon\leq\varepsilon_0.
		\end{equation}
	\end{assumption}
	
	\begin{theorem}\label{linf}
		Let Assumptions \ref{ass1}, \ref{ass2} and \ref{structural} hold.
		Let $u\in \mathcal{U}$ be a control satisfying \eqref{ineq1}. Then,  there is  a constant $\kappa>0$ such that
		\begin{equation}\label{growth}
		J'(u)(v-u)\geq \kappa\|v-u\|_{L^1(\Omega)}^{1+\frac{1}{\gamma}}\quad \forall v\in \mathcal{U},
		\end{equation}
			where 
		\begin{equation}\label{kappa}
		\kappa=\dis \frac{1}{2[4\max(|\alpha|,|\beta|)K]^{\frac{1}{\gamma}}}.
		\end{equation}		
In addition, if $\gamma=+\infty$, then there exists $\eta>0$ such that		
		\begin{equation}\label{cdt2}
			J(v)\geq J(u)+\frac{\kappa}{2}\|v-u\|_{L^1(\Omega)}\quad \forall v\in \mathcal{U}\cap B_{\eta}^\infty(u),
		\end{equation}
where $B_{\eta}^\infty(u)$ is the open ball in $L^\infty(\Omega)$ with center $u$ and radius $\eta$. So, $u$ is strictly locally optimal in the sense of $L^\infty(\Omega)$.
	\end{theorem}
	
	\begin{proof}
		We proceed in two steps.

\textbf{Step 1.} We prove first that there is a constant $\kappa>0$ such that \eqref{growth} holds.
Indeed, from \eqref{ineq1}, we have that for a.e $x\in \Omega$,
$$\left(\frac{\partial L}{\partial u}(x,y(x),u(x))+q(x) \frac{\partial F}{\partial u}(x,y(x),u(x)\right)(v(x)-u(x))\geq 0\;\;\;\text{for every}\;\;\; v\in \U.$$
Now, let $\varepsilon>0$ and define the set 
$$\Omega_{\varepsilon}:=\left\{x\in \Omega: \left|\frac{\partial L}{\partial u}(x,y(x),u(x))+ q(x) \frac{\partial F}{\partial u}(x,y(x),u(x))\right|\geq \varepsilon \right\}.$$
It follows from \eqref{struct} that 
\begin{equation}\label{5}
|\Omega\setminus\Omega_{\varepsilon}|\leq K\varepsilon^{\gamma}.
\end{equation}	
Hence,
\begin{eqnarray}\label{3}
J'(u)(v-u)&=&\dis\int_{\Omega}\left( \frac{\partial L}{\partial u}(x,y(x),u(x))+q(x) \frac{\partial F}{\partial u}(x,y(x),u(x)\right)(v(x)-u(x))\,\dx\nonumber\\
&\geq&\dis\int_{\Omega} \left|\frac{\partial L}{\partial u}(x,y(x),u(x))+ q(x) \frac{\partial F}{\partial u}(x,y(x),u(x))\right||v(x)-u(x)|\,\dx\nonumber\\
&\geq& \varepsilon\|v-u\|_{L^1(\Omega_{\varepsilon})}=\varepsilon\|v-u\|_{L^1(\Omega)}-\varepsilon\|v-u\|_{L^1(\Omega\setminus\Omega_{\varepsilon})}.
\end{eqnarray}	
Thanks to \eqref{5}, we have that
	\begin{equation}\label{4}
	\|v-u\|_{L^1(\Omega\setminus\Omega_{\varepsilon})}=\dis \int_{\Omega\setminus\Omega_{\varepsilon}}|v(x)-u(x)|\,\dx\leq2\max(|\alpha|,|\beta|)K\varepsilon^{\gamma}.
	\end{equation}	
Combining \eqref{3}-\eqref{4} we arrive to 
\begin{equation}\label{6}
J'(u)(v-u)\geq \varepsilon\|v-u\|_{L^1(\Omega)}-2\varepsilon\max(|\alpha|,|\beta|)K\varepsilon^{\gamma}.
\end{equation}
Taking 
$$\dis \varepsilon:=\frac{\|v-u\|_{L^1(\Omega)}^{\frac{1}{\gamma}}}{[4\max(|\alpha|,|\beta|)K]^{\frac{1}{\gamma}}},$$
it follows from \eqref{6} that 
\begin{equation*}
J'(u)(v-u)\geq\frac{\|v-u\|_{L^1(\Omega)}^{1+\frac{1}{\gamma}}}{2[4\max(|\alpha|,|\beta|)K]^{\frac{1}{\gamma}}}.
\end{equation*}
Therefore,  \eqref{growth} holds with $\kappa$ given by \eqref{kappa}.

 \textbf{Step 2.} We prove \eqref{cdt2}. Let $u\in \mathcal{U}\cap B_{\eta}^\infty(u)$, where $\eta$ will be fixed later. Using Proposition \ref{diff5}, we have that $J$ is of class $\mathcal{ C}^2$. Therefore,  the second-order Taylor expansion of $J$ at $u$ gives
$$J(v)=J(u)+J'(u)(v-u)+\frac{1}{2}J''(u+\theta(v-u))(v-u)^2,$$
where $\theta\in (0,1).$
Using \eqref{estJs} and \eqref{growth} with $\gamma=+\infty$,  we obtain that
$$J(v)\geq J(u)+\kappa\|v-u\|_{L^1(\Omega)}-C\|v-u\|^2_{L^2(\Omega)},$$
where the constant $C$ is independent of $u$ and $v$.  Observe  that 
\begin{equation*}
\|v-u\|^2_{L^2(\Omega)}\leq \|v-u\|_{L^\infty}\|v-u\|_{L^1(\Omega)}\leq \eta \|v-u\|_{L^1(\Omega)}.
\end{equation*}
Therefore, 
$$J(v)\geq J(u)+(\kappa-\eta C)\|v-u\|_{L^1(\Omega)}.$$
Taking $\eta=\frac{\kappa}{2C}$ leads us to \eqref{cdt2}. This completes the proof.
	\end{proof}

\subsection{Second order sufficient optimality conditions for $L^2(\Omega)$- local solutions}

The main concern of the present section is to find sufficient conditions for $L^2(\Omega)$-local optimality of a feasible control $\bar u\in\mathcal U$ that satisfies the first order necessary optimality conditions. To do this we need the following result.

\begin{proposition}
Let $0<s<1$ and $\tilde{p}>\frac{N}{2s}$.  Let Assumptions \ref{ass1} and \ref{ass2} hold. Assume that $\Omega$ is of class $\mathcal{C}^{1,\sigma}$ for $0<\sigma:=2s-\frac{N}{{\tilde{p}}}-1<1$.  Then, there are positive constants $C$ and $C_{\tilde p}$ such that for all $u,\bar u\in\mathcal U$ the following estimates hold:
\begin{align}
\|y-\bar y\|_{L^2(\Omega)}&\leq C\|u-\bar u\|_{L^1(\Omega)}\quad \mbox{ if } N<4s\label{i1}\\
\|y-\bar y\|_{L^\infty(\Omega)}&\leq C_{\tilde p}\|u-\bar u\|_{L^{\tilde p}(\Omega)}\label{i2}\\
\|q-\bar q\|_{L^\infty(\Omega)}&\leq C_{\tilde p}\|u-\bar u\|_{L^{\tilde p}(\Omega)} \label{i3}\\
\|G'(u)v\|_{L^2(\Omega)}&\leq C\|v\|_{L^1(\Omega)},\;\qquad \; \forall\; v\in L^\infty(\Omega) \;\mbox{ and } N<4s,\label{i4}
\end{align}
where $y$ and $\bar y$, $q$ and $\bar q$ are the states and adjoint states  associated  to the controls $u$ and $\bar u$, respectively. 
\end{proposition}

\begin{proof}
It follows from \eqref{estc01} that there is a constant $M>0$ such that,
\begin{equation}\label{mj1}
\|y\|_{\mathcal C(\bar\Omega)}+\|u\|_{L^\infty(\Omega)}\le M\;\;\forall u\in\mathcal U.
\end{equation}
Subtracting the equations satisfied by $y$ and $\bar y$ and using the Mean Value Theorem we get 
\begin{align}\label{SJW}
(-\Delta_D)^s(y-\bar y)=& \Big(F(x,y,\bar u)-F(x,\bar y,\bar u)\Big)+ \Big(F(x,y,u)-F(x,y,\bar u)\Big)\notag\\
=&\frac{\partial F}{\partial y}(x,\bar y+\theta(y-\bar y),\bar u)(y-\bar y)+ \Big(F(x,y,u)-F(x,y,\bar u)\Big).
\end{align}
Since $N<4s$, we have that $2<N/(N-2s)$. Thus,  using Theorem \ref{theo-23} and Assumption \ref{ass1}  we can deduce that there is a constant $C_1>0$ such that
\begin{align*}
\|y-\bar y\|_{L^2(\Omega)}\le C_1\|F(x,y,u)-F(x,y,\bar u)\|_{L^1(\Omega)}\le CC_{F,M}\|u-\bar u\|_{L^1(\Omega)}
\end{align*}
and we have shown \eqref{i1}.

Using again Equation \eqref{SJW}, the $L^\infty(\Omega)$-estimates for the linear system given in Theorem \ref{theo-23} and Assumption \ref{ass1}, we obtain that there is a constant $C(\tilde p)>0$ such that
\begin{align*}
\|y-\bar y\|_{L^{\infty}(\Omega)}\le C(\tilde p)\|F(x,y,u)-F(x,y,\bar u)\|_{L^{\tilde p}(\Omega)}\le C(\tilde p)C_{F,M}\|u-\bar u\|_{L^{\tilde p}(\Omega)}
\end{align*}
and this gives \eqref{i2}.

To prove \eqref{i3},  we subtract the equations satisfied by $q$ and $\bar q$ to get
\begin{align*}
(-\Delta_D)^s(q-\bar q)=&\frac{\partial F}{\partial y}(x,y,u)q-\frac{\partial F}{\partial y}(x,\bar y,\bar u)\bar q  +\frac{\partial L}{\partial y}(x,y,u)-\frac{\partial L}{\partial y}(x,\bar y,\bar u)\\
=&\frac{\partial F}{\partial y}(x,y,u)(q-\bar q)+\left[\frac{\partial F}{\partial y}(x,y,u)-\frac{\partial F}{\partial y}(x,\bar y,\bar u)\right]q+\left[\frac{\partial L}{\partial y}(x,y,u)-\frac{\partial L}{\partial y}(x,\bar y,\bar u)\right].
\end{align*}
Using the $L^\infty(\Omega)$-estimates given in Theorem \ref{theo-23}, Assumptions \ref{ass1} and \ref{ass2},  and \eqref{i2} we can deduce that there is a constant $C(\tilde p)>0$ such that
\begin{align*}
\|q-\bar q \|_{L^\infty(\Omega)}\le& C(\tilde p)\left[C_{F,N}M\|y-\bar y\|_{L^{\tilde p}(\Omega)}+\|u-\bar u\|_{L^{\tilde p}(\Omega)}+C_{L,M}\|y-\bar y\|_{L^{\tilde p}(\Omega)}\right]\\
\le &C(\tilde p)\|u-\bar u\|_{L^{\tilde p}(\Omega)}
\end{align*}
and we have shown \eqref{i3}.
Finally, \eqref{i4} is a direct consequence of Assumption \ref{ass1}, Theorem \ref{theo-23} and \eqref{diff1}. The proof is finished.
\end{proof}

The following lemma will be very useful for the rest of the paper.

\begin{lemma}\cite[\bf Lemma 3.5]{casas2012}
	\label{analysis}
Let $(X,\Sigma,\mu)$ be a measure space with $\mu(X)<+\infty$. Suppose that $(g_k)_{k\geq 1}\subset L^\infty(X)$ and $(v_k)_{k\geq 1}\subset L^2(X)$ satisfy the following assumptions:
\begin{itemize}
\item $g_k\geq 0$ a.e. in $X$, $\forall k\geq 1$, $(g_k)_{k\geq 1}$ is bounded in $L^\infty(X)$ and $g_k\to g$ strongly in $L^1(\Omega)$,  as $k\to \infty$.
\item $v_k\rightharpoonup v$ weakly in $L^2(\Omega)$,  as $k\to \infty$.
\end{itemize}
Then,
\begin{equation}
\int_{X}g(x)v^2(x) d\mu(x)\leq \liminf_{k\to \infty}\int_{X}g_k(x)v_k^2(x) d\mu(x).
\end{equation}
\end{lemma}

In the rest of the paper, we assume that 
\begin{equation}\label{cond}
	\exists \nu>0 \;\;\text{such that}\;\; \frac{\partial ^2 H}{\partial u^2}(x,y(x),q(x),u(x))\geq \nu\quad\text{for a.e.}\;\; x\in \Omega,\;\;\forall  u(x)\in [\alpha,\beta].
\end{equation}
where $y$ and $q$ are the state and adjoint state associated with $u$, respectively, and $H$ is the Hamiltonian given in \eqref{ham}.

\begin{remark} The second Fr\'echet derivative of $J$ can be rewritten in terms of the Hamiltonian function $H$ defined in \eqref{ham} as follows:
\begin{equation}\label{diff7}
\left.
\begin{array}{lllll}
J''(u)v^2 &=&\dis \int_{\Omega} \frac{\partial^2 H}{\partial y^2}(x,y,q,u)(G'(u)v)^2\,\dx + 2\int_{\Omega} \frac{\partial^2 H}{\partial y\partial u}(x,y,q,u)(vG'(u)v)\,\dx\\
& &\dis +\int_{\Omega} \frac{\partial^2 H}{\partial u^2}(x,y,q,u)v^2\,\dx.
\end{array}
\right.
\end{equation}	
\end{remark}

\begin{theorem}
Let $u\in \U$ satisfy the first-order optimality condition \eqref{ineq} and $y$, $q$ be the associated state and adjoint state,  respectively.  Let also Assumptions \ref{ass1} and \ref{ass2} hold and $N<4s$. We assume that the structural assumption \eqref{struct} holds with $\gamma>1$. Then,  there exists $\varepsilon>0$ such that 
\begin{equation}\label{cdt3}
J(v)\geq J(u)+\frac{\kappa}{2}\|v-u\|^{1+\frac{1}{\gamma}}_{L^1(\Omega)}+\frac{\nu}{8}\|v-u\|^2_{L^2(\Omega)}\quad \forall v\in \mathcal{U}\cap \bar{B}_{\varepsilon}^2(u),
\end{equation}
where $\bar{B}_{\varepsilon}^2(u)$ is the closed ball in $L^2(\Omega)$ with center $u$ and radius $\varepsilon$ and $\kappa$ is given in \eqref{kappa}. 
\end{theorem}	

\begin{proof}
Let $ v\in \mathcal{U}\cap \bar{B}_{\varepsilon}^2(u)$, where $\varepsilon$ will be chosen later. We set $v_\theta:=u+\theta(v-u)$, $\theta\in (0,1)$. Using the Taylor expansion and \eqref{growth},   we write
\begin{align}\label{diff8}
J(v) =&\dis J(u)+J'(u)(v-u)+\frac{1}{2}J''(v_\theta)(v-u)^2\notag\\
\geq&\dis J(u)+ \kappa\|v-u\|_{L^1(\Omega)}^{1+\frac{1}{\gamma}}+\frac{1}{2}J''(v_\theta)(v-u)^2.
\end{align}	
We have that,
\begin{align}\label{diff9}
J''(v_\theta)(v-u)^2 &=\dis \int_{\Omega} \frac{\partial^2 H}{\partial y^2}(x,y_\theta,q_\theta,v_\theta)(G'(v_\theta)(v-u))^2\,\dx \notag\\
&\dis + 2\int_{\Omega} \frac{\partial^2 H}{\partial y\partial u}(x,y_\theta,q_\theta,v_\theta)(v-u)G'(v_\theta)(v-u)\,\dx\notag\\
 &\dis +\int_{\Omega} \frac{\partial^2 H}{\partial u^2}(x,y_\theta,q_\theta,v_\theta)(v-u)^2\,\dx,
\end{align}	
where $y_\theta$ and $q_{\theta}$ are the state and the adjoint state associated to $v_\theta$,  respectively. We note that from \eqref{estc01} and \eqref{estadjoint}, we have that there exists $M>0$ such that
\begin{equation}\label{i0}
\|y_\theta\|_{\mathcal{C}(\bar{\Omega})}+\|q_\theta\|_{\mathcal{C}(\bar{\Omega})}\leq M.
\end{equation}
Using Young's inequality, \eqref{as2}, \eqref{as7}, \eqref{i4} and \eqref{i0} we obtain that there is a constant $C>0$ such that
\begin{equation}\label{diff10}
\left.
\begin{array}{lllll}
\dis\left|\int_{\Omega} \frac{\partial^2 H}{\partial y^2}(x,y_\theta,q_\theta,v_\theta)(G'(v_\theta)(v-u))^2\,\dx\right|&\leq & C\|v-u\|^2_{L^1(\Omega)},
\end{array}
\right.
\end{equation}	
and 
\begin{align}\label{diff11}
\dis 2\left|\int_{\Omega} \frac{\partial^2 H}{\partial y\partial u}(x,y_\theta,q_\theta,v_\theta)(v-u)G'(v_\theta)(v-u)\,\dx\right|\leq & \dis C\int_{\Omega}\left| (v-u)G'(v_\theta)(v-u)\right|\,\dx\notag\\
\leq& \dis \frac{\nu}{4}\|v-u\|^2_{L^2(\Omega)}+C\|v-u\|^2_{L^1(\Omega)}.
\end{align}
For the last term of \eqref{diff9}, we write
\begin{align}\label{diff12}
\dis \dis \int_{\Omega} \frac{\partial^2 H}{\partial u^2}(x,y_\theta,q_\theta,v_\theta)(v-u)^2\,\dx= &\dis  \int_{\Omega} \left[\frac{\partial^2 L}{\partial u^2}(x,y_\theta,v_\theta)-\frac{\partial^2 L}{\partial u^2}(x,y,v_\theta)\right](v-u)^2\,\dx\notag\\
& +\dis  \int_{\Omega} q\left[\frac{\partial^2 F}{\partial u^2}(x,y_\theta,v_\theta)-\frac{\partial^2 F}{\partial u^2}(x,y,v_\theta)\right](v-u)^2\,\dx\notag\\
& +\dis  \int_{\Omega} (q_{\theta}-q)\frac{\partial^2 F}{\partial u^2}(x,y_\theta,v_\theta)(v-u)^2\,\dx\notag\\
&+\dis  \int_{\Omega} \frac{\partial^2 H}{\partial u^2}(x,y,q,v_\theta)(v-u)^2\,\dx.
\end{align}	
Using \eqref{as2}, \eqref{as3}, \eqref{as8}, \eqref{i0} and \eqref{i3} with $\tilde{p}=2$, we get
\begin{equation}\label{diff13}
\left.
\begin{array}{lllll}
\dis \left|\int_{\Omega}\left[\frac{\partial^2 L}{\partial u^2}(x,y_\theta,v_\theta)-\frac{\partial^2 L}{\partial u^2}(x,y,v_\theta)\right](v-u)^2\,\dx\right|\leq \varepsilon \|v-u\|^2_{L^2(\Omega)},
\end{array}
\right.
\end{equation}	
\begin{equation}\label{diff14}
\left.
\begin{array}{lllll}
\dis \left|\int_{\Omega} q\left[\frac{\partial^2 F}{\partial u^2}(x,y_\theta,v_\theta)-\frac{\partial^2 F}{\partial u^2}(x,y,v_\theta)\right](v-u)^2\,\dx\right|\leq \varepsilon M \|v-u\|^2_{L^2(\Omega)},
\end{array}
\right.
\end{equation}	
\begin{equation}\label{diff15}
\left.
\begin{array}{lllll}
\dis \left|\int_{\Omega} (q_{\theta}-q)\frac{\partial^2 F}{\partial u^2}(x,y_\theta,v_\theta)(v-u)^2\,\dx\right|\leq C_{F,M} \|v-u\|^3_{L^2(\Omega)}.
\end{array}
\right.
\end{equation}	
Choosing $\varepsilon>0$ sufficiently small such that $\|v-u\|_{L^2(\Omega)}<\varepsilon$ and $\varepsilon(1+M+C_{F,M})\leq \frac{\nu}{2}$, we can deduce from \eqref{diff12}, \eqref{diff15} and \eqref{cond} that 
\begin{align}\label{diff16}
\dis \dis \int_{\Omega} \frac{\partial^2 H}{\partial u^2}(x,y_\theta,q_\theta,v_\theta)(v-u)^2\,\dx\geq & \dis -\frac{\nu}{2}\|v-u\|^2_{L^2(\Omega)}+\dis  \int_{\Omega} \frac{\partial^2 H}{\partial u^2}(x,y,q,v_\theta)(v-u)^2\,\dx\notag\\
\geq & \dis -\frac{\nu}{2}\|v-u\|^2_{L^2(\Omega)}+\nu\|v-u\|^2_{L^2(\Omega)}\notag\\
=&\dis \frac{\nu}{2}\|v-u\|^2_{L^2(\Omega)}.
\end{align}
Combining \eqref{diff9}, \eqref{diff10}, \eqref{diff11} and \eqref{diff16}, we arrive to
\begin{equation}\label{diff17}
J''(v_\theta)(v-u)^2 \geq \dis\frac{\nu}{4}\|v-u\|^2_{L^2(\Omega)}-2C\|v-u\|^2_{L^1(\Omega)}.
\end{equation}	
Hence, \eqref{diff8} leads us to
\begin{equation}\label{diff18}
J(v) \geq\dis J(u)+ \kappa\|v-u\|_{L^1(\Omega)}^{1+\frac{1}{\gamma}}-C\|v-u\|^2_{L^1(\Omega)}+\frac{\nu}{8}\|v-u\|^2_{L^2(\Omega)}.
\end{equation}	
Now, since $\gamma>1$, we have that
\begin{equation}\label{diff19}
\kappa\|v-u\|_{L^1(\Omega)}^{1+\frac{1}{\gamma}}-C\|v-u\|^2_{L^1(\Omega)}= \|v-u\|_{L^1(\Omega)}^{1+\frac{1}{\gamma}}\left[\kappa-C\|v-u\|_{L^1(\Omega)}^{1-\frac{1}{\gamma}}\right].
\end{equation}	
So, if we choose $\varepsilon>0$ such that 
\begin{equation*}
C\|v-u\|_{L^1(\Omega)}^{1-\frac{1}{\gamma}}\leq C|\Omega|^{\frac{1}{2}-\frac{1}{2\gamma}}\|v-u\|_{L^2(\Omega)}^{1-\frac{1}{\gamma}}\leq C|\Omega|^{\frac{1}{2}-\frac{1}{2\gamma}}\varepsilon^{1-\frac{1}{\gamma}}\leq \frac{\kappa}{2},
\end{equation*}
we can deduce \eqref{cdt3} from \eqref{diff19}. The proof is finished.
\end{proof}	

Now, we can state and prove the last main result of the paper.

\begin{theorem}
	Let $u\in \U$ satisfy the first-order optimality condition \eqref{ineq} and $y$, $q$ be the associated state and adjoint state,  respectively.  Let also Assumptions \ref{ass1} and \ref{ass2} hold. We assume that 
	\begin{equation}\label{cdt1}
	{J}''(u)v^2> 0 \quad \forall v\in \mathcal C_u\backslash\{0\}.
	\end{equation}
	Then,  there exist $\varepsilon>0$ and $\delta>0$ such that 
	\begin{equation}\label{cdt4}
	J(v)\geq J(u)+\frac{\delta}{2}\|v-u\|^2_{L^2(\Omega)}\quad \forall v\in \mathcal{U}\cap \bar{B}_{\varepsilon}^2(u),
	\end{equation}
	where $\bar{B}_{\varepsilon}^2(u)$ is the closed ball in $L^2(\Omega)$ with center $u$ and radius $\varepsilon$. 
\end{theorem}	

\begin{proof}
We argue by contradiction by assuming that there exists a sequence $(v_k)_{k\geq 1}\subset \U$ such that
\begin{equation}\label{cdt5}
\|v_k-u\|_{L^2(\Omega)}<\frac{1}{k}\quad \text{and}\quad J(v_k)< J(u)+\frac{1}{2k}\|v_k-u\|^2_{L^2(\Omega)}\quad \forall k\geq 1.
\end{equation}
Let us define 
$$\rho_k:=\|v_k-u\|_{L^2(\Omega)}\quad\text{and}\quad w_k:=\frac{1}{\rho_k}(v_k-u) \quad \text{for}\;\; k\geq 1.$$
Since $\|w_k\|_{L^2(\Omega)}=1$, we can take a subsequence, denoted in the same way, such that $w_k\rightharpoonup w$ weakly in $L^2(\Omega)$, as $k\to\infty$.  Let us prove that $w\in \mathcal C_u$. We first note that $w_k$ satisfies \eqref{eq1}. Since the set of functions in $L^2(\Omega)$ satisfying \eqref{eq1} is closed and convex, it is weakly closed  and so $w$ satisfies \eqref{eq1}. Let us show that $J'(u)w=0$. From the optimality condition \eqref{ineq1}, we can deduce that
\begin{equation}\label{k1}
\dis J'(u)w=\lim_{k\to \infty}J'(u)w_k=\lim_{k\to \infty}\frac{1}{\rho_k}\int_{\Omega}\frac{\partial H}{\partial u}(x,y,q,u)(v_k-u)\,\dx\geq 0.
\end{equation}
From \eqref{cdt5} we have that, 
\begin{equation}\label{k4}
\left.
\begin{array}{lllll}
\dis J(v_k)-J(u)<\frac{1}{2k}\|v_k-u\|^2_{L^2(\Omega)}=\frac{\rho^2_k}{2k}.
\end{array}
\right.
\end{equation}	
On the other hand, we write the first-order expansion
\begin{equation*}
\left.
\begin{array}{lllll}
J(v_k)-J(u)&=&J'(v_{\theta_k})(v_k-u)\\
&=&\rho_kJ'(v_{\theta_k})w_k\\
&=& \dis \rho_k\int_{\Omega} \left(\frac{\partial L}{\partial u}(x,y_{\theta_k},v_{\theta_k})+ q_{\theta_k}\frac{\partial F}{\partial u}(x,y_{\theta_k},v_{\theta_k}) \right) w_k\,\dx,
\end{array}
\right.
\end{equation*}	
where $\theta_k\in (0,1)$, $v_{\theta_k}:=u+\theta_k(v_k-u)$, $y_{\theta_k}$ and $q_{\theta_k}$ are the state and the adjoint state associated to $v_{\theta_k}$, respectively.  Therefore, using \eqref{k4} we arrive to 
\begin{equation}\label{k2}
\left.
\begin{array}{lllll}
\dis J'(v_{\theta_k})w_k=\int_{\Omega} \left(\frac{\partial L}{\partial u}(x,y_{\theta_k},v_{\theta_k})+ q_{\theta_k}\frac{\partial F}{\partial u}(x,y_{\theta_k},v_{\theta_k}) \right) w_k\,\dx\leq \frac{\rho_k}{2k},\; \quad \forall k\geq 1.
\end{array}
\right.
\end{equation}	
Moreover, using \eqref{cdt5}, we have that,   
\begin{equation}\label{k3}
\dis\lim_{k\to\infty} \|v_{\theta_k}-u\|_{L^2(\Omega)}\leq \lim_{k\to\infty} \|v_{k}-u\|_{L^2(\Omega)}\le \lim_{k\to\infty}\frac{1}{k}= 0.
\end{equation}	
On the one hand, using the Mean Value Theorem, \eqref{as2} and \eqref{as7}, we get after some calculations, 
\begin{equation*}
\left\|\frac{\partial L}{\partial u}(x,y_{\theta_k},v_{\theta_k})-\frac{\partial L}{\partial u}(x,y,u)\right\|_{L^2(\Omega)}\leq C_{L,M}\left[\|y_{\theta_k}-y\|_{L^2(\Omega)}+\|v_{\theta_k}-u\|_{L^2(\Omega)}\right]
\end{equation*}
and 
\begin{equation*}
\left\|\frac{\partial F}{\partial u}(x,y_{\theta_k},v_{\theta_k})-\frac{\partial F}{\partial u}(x,y,u)\right\|_{L^2(\Omega)}\leq C_{F,M}\left[\|y_{\theta_k}-y\|_{L^2(\Omega)}+\|v_{\theta_k}-u\|_{L^2(\Omega)}\right].
\end{equation*}
Therefore, passing to the limit,  as $k\to \infty$, we arrive to 
\begin{equation}\label{estL1}
\frac{\partial L}{\partial u}(x,y_{\theta_k},v_{\theta_k})\to \frac{\partial L}{\partial u}(x,y,u) \quad\text{strongly in}\quad L^2(\Omega)
\end{equation}
and 
\begin{equation}\label{estF1}
\frac{\partial F}{\partial u}(x,y_{\theta_k},v_{\theta_k})\to \frac{\partial F}{\partial u}(x,y,u) \quad\text{strongly in}\quad L^2(\Omega).
\end{equation}
On the other hand, using again \eqref{k3}, we can deduce from \eqref{i3} with $\tilde{p}=2$ that, as $k\to\infty$, 
\begin{equation}\label{k5}
q_{\theta_k}\to q \quad\text{strongly in}\quad L^\infty(\Omega).
\end{equation}
Taking the limit in \eqref{k2},  as $k\to \infty$,  while using \eqref{estL1}-\eqref{k5} along with the weak convergence of $w_k$ and the weak-strong convergence, we can deduce that
$J'(u)w\leq 0$. Hence, using \eqref{k1}, we have that $J'(u)w=0$. Thus,  $w\in \mathcal C_u$.

Next, we show that $J''(u)w\leq 0$. 
Thanks to the second-order expansion and using again \eqref{ineq} and \eqref{k4}, we obtain that 
\begin{equation*}
\left.
\begin{array}{lllll}
\dis \frac{\rho^2_k}{2k}&=&J(v_k)-J(u)\\
&=& J'(u)(v_k-u)+\frac{1}{2}J''(u+\eta_k(v_k-u))(v_k-u)^2\\
&\geq& \dis \frac{1}{2}J''(u+\eta_k(v_k-u))(v_k-u)^2\\
&=& \dis \frac{\rho^2_k}{2}J''(u+\eta_k(v_k-u))w_k^2
\end{array}
\right.
\end{equation*}	
with $\eta_k\in (0,1)$.
Dividing both members of the above inequality by $\frac{\rho^2_k}{2}$, we arrive to 
\begin{equation}\label{k6}
\left.
\begin{array}{lllll}
\dis J''(u+\eta_k(v_k-u))w_k^2\leq \frac{1}{k} \quad \forall k\geq 1.
\end{array}
\right.
\end{equation}	
Set $v_{\eta_k}:=u+\eta_k(v_k-u)$ and denote by $y_{\eta_k}$ and $q_{\eta_k}$  the state and the adjoint state associated to $v_{\eta_k}$,  respectively. 
Using \eqref{cdt5}, we have that,  as $k\to \infty$,
\begin{equation}\label{k9}
\left.
\begin{array}{lllll}
\dis \|v_{\eta_k}-u\|_{L^2(\Omega)}\to 0.
\end{array}
\right.
\end{equation}
Therefore, from Proposition \ref{convergence} and \eqref{i3} (with $\tilde{p}=2$), we have that, as $k\to\infty$, 
\begin{equation}\label{k10}
y_{\eta_k}\to y\;\; \text{strongly in}\;\;\; \V\cap \mathcal{C}^{0,\sigma}(\bar{\Omega})
\end{equation}
and  
\begin{equation}\label{k11}
q_{\eta_k}\to q \quad\text{strongly in}\quad L^\infty(\Omega).
\end{equation}
In addition, using the weak convergence of $w_k$ in $L^2(\Omega)$, one can show that (up to a subsequence, if needed), as $k\to\infty$, 
\begin{equation}\label{k12}
G'(v_{\eta_k})w_k\to G'(u)w\;\; \text{strongly in}\;\;\; L^2(\Omega).
\end{equation}
Now, using \eqref{diff7}, we have that, 
\begin{align}\label{k7}
J''(v_{\eta_k})w_k^2 =&\dis \int_{\Omega} \frac{\partial^2 H}{\partial y^2}(x,y_{\eta_k},q_{\eta_k},v_{\eta_k})(G'(v_{\eta_k})w_k)^2\,\dx\notag \\
& \dis+ 2\int_{\Omega} \frac{\partial^2 H}{\partial y\partial u}(x,y_{\eta_k},q_{\eta_k},v_{\eta_k})(w_kG'(v_{\eta_k})w_k)\,\dx\notag\\
 &\dis +\int_{\Omega} \frac{\partial^2 H}{\partial u^2}(x,y_{\eta_k},q_{\eta_k},v_{\eta_k})w_k^2\,\dx.
\end{align}	

Next, we claim that, as $k\to\infty$, 
\begin{equation}\label{k14}
\frac{\partial^2 H}{\partial y^2}(x,y_{\eta_k},q_{\eta_k},v_{\eta_k})\to \frac{\partial^2 H}{\partial y^2}(x,y,q,u)\;\; \text{strongly in}\;\;\; L^\infty(\Omega)
\end{equation}
and
\begin{equation}\label{k15}
\frac{\partial^2 H}{\partial y\partial u}(x,y_{\eta_k},q_{\eta_k},v_{\eta_k})\to \frac{\partial^2 H}{\partial y\partial u}(x,y,q,u)\;\; \text{strongly in}\;\;\; L^\infty(\Omega).
\end{equation}
Indeed, using \eqref{as2}, \eqref{as3} and \eqref{as8} with $\varepsilon=\frac{1}{k}$, we obtain
\begin{equation*}
\begin{array}{lllll}
\dis \left\|\frac{\partial^2 H}{\partial y^2}(x,y_{\eta_k},q_{\eta_k},v_{\eta_k})- \frac{\partial^2 H}{\partial y^2}(x,y,q,u)\right\|_{L^\infty(\Omega)}&\leq& \dis\left\|\frac{\partial^2 L}{\partial y^2}(x,y_{\eta_k},v_{\eta_k})- \frac{\partial^2 L}{\partial y^2}(x,y,u)\right\|_{L^\infty(\Omega)}\\
&&\dis +\|q_{\eta_k}-q\|_{L^\infty(\Omega)}\left\|\frac{\partial^2 F}{\partial y^2}(x,y_{\eta_k},v_{\eta_k})\right\|_{L^\infty(\Omega)}\\
&&\dis + \|q\|_{L^\infty(\Omega)}\left\|\frac{\partial^2 F}{\partial y^2}(x,y_{\eta_k},v_{\eta_k})- \frac{\partial^2 F}{\partial y^2}(x,y,u)\right\|_{L^\infty(\Omega)}\\
&\leq& \dis  \frac{1}{k}(1+\|q\|_{L^\infty(\Omega)})+ C_{F,M}\|q_{\eta_k}-q\|_{L^\infty(\Omega)}.
\end{array}
\end{equation*}
Taking the limit,  as $k\to \infty$,  in the last inequality, while using \eqref{k11} leads us to \eqref{k14}. With similar arguments, we get \eqref{k15}. Hence, combining \eqref{k12}, \eqref{k14}, \eqref{k15} and thanks to the weak convergence $w_k\rightharpoonup w$ in $L^2(\Omega)$,  as $k\to\infty$, we can  deduce that, 
\begin{align}\label{k16}
&\dis\lim_{k\to \infty} \left[\int_{\Omega} \frac{\partial^2 H}{\partial y^2}(x,y_{\eta_k},q_{\eta_k},v_{\eta_k})(G'(v_{\eta_k})w_k)^2\,\dx + 2\int_{\Omega} \frac{\partial^2 H}{\partial y\partial u}(x,y_{\eta_k},q_{\eta_k},v_{\eta_k})(w_kG'(v_{\eta_k})w_k)\,\dx\right]\notag\\
=&\dis \int_{\Omega} \frac{\partial^2 H}{\partial y^2}(x,y,q,u)(G'(u)w)^2\,\dx + 2\int_{\Omega} \frac{\partial^2 H}{\partial y\partial u}(x,y,q,u)(wG'(u)w)\,\dx.
\end{align}	
 In addition, we can write
\begin{equation*}
\begin{array}{lllll}
\dis \frac{\partial^2 H}{\partial u^2}(x,y_{\eta_k},q_{\eta_k},v_{\eta_k})&=& \dis\frac{\partial^2 L}{\partial u^2}(x,y_{\eta_k},v_{\eta_k})- \frac{\partial^2 L}{\partial u^2}(x,y,v_{\eta_k})\\
&&\dis +(q_{\eta_k}-q)\frac{\partial^2 F}{\partial u^2}(x,y_{\eta_k},v_{\eta_k})\\
&&\dis + q\left[\frac{\partial^2 F}{\partial u^2}(x,y_{\eta_k},v_{\eta_k})- \frac{\partial^2 F}{\partial u^2}(x,y,v_{\eta_k})\right]\\
&&+\dis \frac{\partial^2 H}{\partial u^2}(x,y,q,v_{\eta_k}).
\end{array}
\end{equation*}
Using \eqref{as2}, \eqref{as3} and \eqref{as8} with $\varepsilon=\frac{1}{k}$, \eqref{i3} with $\tilde{p}=2$, and letting $k\to \infty$, we arrive to 
\begin{align}\label{k18}
 &\dis\left|\frac{\partial^2 L}{\partial u^2}(x,y_{\eta_k},v_{\eta_k})- \frac{\partial^2 L}{\partial u^2}(x,y,v_{\eta_k})\right| \leq \frac{1}{k}\to 0\notag\\
&\dis \left|(q_{\eta_k}-q)\frac{\partial^2 F}{\partial u^2}(x,y_{\eta_k},v_{\eta_k})\right|\leq C_{F,M}\|q_{\eta_k}-q\|_{L^\infty(\Omega)}\to 0\notag\\
&\dis \left|q\left[\frac{\partial^2 F}{\partial u^2}(x,y_{\eta_k},v_{\eta_k})- \frac{\partial^2 F}{\partial u^2}(x,y,v_{\eta_k})\right]\right|\leq \frac{1}{k}\|q\|_{L^\infty(\Omega)}\to 0.
\end{align}
Now,  letting $g_k:=\dis \frac{\partial^2 H}{\partial u^2}(x,y,q,v_{\eta_k})$, we have from \eqref{cond} that $g_k\geq 0$ a.e. in $\Omega$. Using \eqref{as2} and \eqref{as7}, we have that the sequence $(g_k)_{k\geq 1}$ is bounded in $L^\infty(\Omega)$. In addition, if we let $g:=\dis \frac{\partial^2 H}{\partial u^2}(x,y,q,u)$, then applying again \eqref{as3} and \eqref{as8} with $\varepsilon=\frac{1}{k}$ we obtain that,  as $k\to\infty$, 
\begin{equation*}
\begin{array}{lllll}
\dis \int_{\Omega}|g_k(x)-g(x)|\, \dx\leq \frac{|\Omega|}{k}(1+\|q\|_{L^\infty(\Omega)})\to 0.
\end{array}
\end{equation*}
Therefore,  $g_k\to g$ strongly in $L^1(\Omega)$, as $k\to\infty$. Since, $w_k\rightharpoonup w$ weakly in $L^2(\Omega)$,  as $k\to\infty$, we  can deduce from Lemma \ref{analysis} that 
\begin{equation}\label{k17}
\begin{array}{lllll}
\dis \int_{\Omega}\frac{\partial^2 H}{\partial u^2}(x,y,q,u)w^2\,\dx\leq \liminf_{k\to \infty}\int_{\Omega}\frac{\partial^2 H}{\partial u^2}(x,y,q,v_{\eta_k})w_k^2\,\dx.
\end{array}
\end{equation}
Combining \eqref{k18}-\eqref{k17}, we get
\begin{equation}\label{k20}
\dis \int_{\Omega}\frac{\partial^2 H}{\partial u^2}(x,y,q,u)w^2\,\dx\leq \liminf_{k\to \infty}\int_{\Omega}\frac{\partial^2 H}{\partial u^2}(x,y,q_{\eta_k},v_{\eta_k})w_k^2\,\dx.
\end{equation}
The inequality \eqref{k20} together with \eqref{k6}, \eqref{k7}  and  \eqref{k16} leads us to
\begin{equation}\label{k21}
J''(u)w^2\leq \liminf_{k\to \infty}J''(v_{\eta_k})w_k^2\leq 0.
\end{equation}
But,  according to \eqref{cdt1}, this is only possible if $w=0$. Therefore, $w_k\rightharpoonup 0$ weakly in $L^2(\Omega)$ and $G'(v_{\eta_k})w_k\to 0$ strongly in $L^2(\Omega)$, as $k\to\infty$. Thus,  
\begin{equation*}
\left.
\begin{array}{lllll}
\dis\lim_{k\to \infty} \left[\int_{\Omega} \frac{\partial^2 H}{\partial y^2}(x,y_{\eta_k},q_{\eta_k},v_{\eta_k})(G'(v_{\eta_k})w_k)^2\,\dx + 2\int_{\Omega} \frac{\partial^2 H}{\partial y\partial u}(x,y_{\eta_k},q_{\eta_k},v_{\eta_k})(w_kG'(v_{\eta_k})w_k)\,\dx\right]=0.
\end{array}
\right.
\end{equation*}	
Finally, combining \eqref{cond} with the latter convergence, \eqref{k18}, \eqref{k17}, \eqref{k7}, \eqref{k21} and thanks to the fact that $\|w_k\|_{L^2(\Omega)}=1$, we arrive to
\begin{equation*}
\left.
\begin{array}{lllll}
\dis \nu=\nu \liminf_{k\to \infty}\|w_k\|^2_{L^2(\Omega)}&\leq& \dis \int_{\Omega}\frac{\partial^2 H}{\partial u^2}(x,y,q,v_{\eta_k})w_k^2\,\dx\\
&\leq&\dis \liminf_{k\to \infty}\int_{\Omega}\frac{\partial^2 H}{\partial u^2}(x,y,q_{\eta_k},v_{\eta_k})w_k^2\,\dx\\
&\leq& J''(v_{\eta_k})w_k^2\leq 0.
\end{array}
\right.
\end{equation*}	
Hence, $\nu\le 0$, which contradicts the fact that $\nu>0$. This completes the proof.
\end{proof}

\bibliographystyle{abbrv}
\bibliography{references}
	
\end{document}